\documentclass[11pt,reqno]{amsart}
\usepackage{amssymb, amsmath, amsthm,amsrefs,graphicx,marvosym,wasysym}
\usepackage{pstricks}
\usepackage{latexsym}
\usepackage{color}
\usepackage{exscale}
\usepackage{bm, bbm, mathrsfs}
\usepackage{hyperref}

\usepackage{mathtools}

\newtheorem*{theorem*}{Theorem}
\newtheorem*{lemma*}{Lemma}

\newtheorem{theorem}{Theorem}
\newtheorem{lemma}{Lemma}[section]

\newtheorem{remark}[lemma]{Remark}

\newtheorem{claim}[lemma]{Claim}

\newcommand{\RR}{\mathbb R}

\newcommand{\C}{c}

\def\vp{\varphi}

\def\s0{{ s_0}}

\def\ts0{{\tilde s_0}}

\def\eq#1{(\ref{#1})}
\def\nn{\nonumber}

\def\({\left(\begin{array}{cccccc}}
\def\){\end{array}\right)}

\def\bes{\begin{eqnarray}}
\def\ees{\end{eqnarray}}

\newcommand{\beq}{\begin{equation}}
\newcommand{\eeq}{\end{equation}}
\newcommand{\bea}{\begin{eqnarray}}
\newcommand{\eea}{\end{eqnarray}}
\newcommand{\baln}{\begin{align}}
\newcommand{\ealn}{\end{align}}
\newcommand{\beann}{\begin{eqnarray*}}
\newcommand{\eeann}{\end{eqnarray*}}

\newcommand{\iu}[1]{u^{(#1)}}
\newcommand{\ip}[1]{p^{(#1)}}
\newcommand{\iq}[1]{q^{(#1)}}
\newcommand{\ith}[1]{\theta^{(#1)}}

\newcommand{\iv}[1]{v^{(#1)}}

\def\OAC{\overline{OAC}}
\def\OAB{\overline{OAB}}
\def\ABC{\overline{ABC}}
\def\AC{\overline{AC}}
\def\AB{\overline{AB}}
\def\OA{\overline{OA}}
\def\OB{\overline{OB}}

\newcommand{\pf}{\begin{proof}}
\newcommand{\foorp}{\end{proof}}
\newcommand{\nquad}{\negthickspace\negthickspace
\negthickspace\negthickspace}

\begin{document}

\title{A mixed boundary value problem for $u_{xy}=f(x,y,u,u_x,u_y)$}
\author{Helge Kristian Jenssen }\address{ H.~K.~Jenssen, Department of
Mathematics, Penn State University,
University Park, State College, PA 16802, USA ({\tt
jenssen@math.psu.edu}).}
\author{Irina A.\ Kogan}\address{I.~A.~Kogan, Department of Mathematics, North Carolina State University,
    Raleigh,  NC, 27695-8205, USA \texttt{(iakogan@ncsu.edu)}}
\date{\today}

\

\begin{abstract}
Consider a single hyperbolic PDE $u_{xy}=f(x,y,u,u_x,u_y)$, with locally prescribed data: $u$ along a non-characteristic curve $M$ and $u_x$ along a non-characteristic curve $N$. We assume that $M$ and $N$ are graphs of one-to-one functions, intersecting only at the origin, and located in the first quadrant of the $(x,y)$-plane. 

It is known that if $M$ is located above $N$, then there is a unique local solution, obtainable by successive approximation.  We show that in the opposite case, when  $M$ lies below $N$, the uniqueness can fail in the following strong sense: for the same boundary data, there are two solutions that differ at points arbitrarily close to the origin.

In the latter case, we also establish existence of a local solution (under a Lipschitz condition on the function $f$). The construction, via Picard iteration, makes use of a careful choice of additional $u$-data
which are updated in each iteration step.
\end{abstract}

\maketitle


\section{Introduction}\label{intro}
We consider existence and uniqueness for a certain 
type of boundary value problem for 
the second order wave equation
\beq\label{eqn}
	u_{xy}=f(x,y,u,u_x,u_y).
\eeq
(All quantities and variables are real valued.) 
Our main goal is to 
draw attention to an issue related to uniqueness 
of solutions when the data prescribe the unknown itself along 
one non-characteristic curve, together with one of its partial 
derivatives, $u_x$ say, along a different non-characteristic curve.
The curves are assumed to be located in the first quadrant,
both passing through the origin, but otherwise disjoint.
We shall see that,  depending on the relative position of the data curves, 
uniqueness may fail. We also show how a non-standard Picard iteration 
scheme yields existence of a local solution. 
For this we join the two given data 
curves by a third one, along which we iteratively prescribe the 
values of the solution 
itself.

Our motivation stems from our earlier work \cite{bjk} which provided 
a generalization of an integrability theorem due to Darboux.
Darboux's original results concerned (possibly overdetermined) 
systems of first order PDEs where each equation contains a single
partial derivative, for which it is solved, and where the data for the unknowns 
are prescribed locally along certain affine subspaces.  In \cite{bjk} we generalized 
Darboux's results to cases where the data are given along more general 
manifolds and where partial derivatives are replaced by differentiation
along vector fields. However, to establish existence of a unique local solution,
we found it necessary to assume a certain {\em Stable Configuration 
Condition} (SCC). This is a geometric condition on the relative location 
of the data manifolds for the unknowns. As such it has nothing to do 
with over-determinacy or nonlinearity of the PDE system under consideration.

Therefore, to investigate the necessity of the SCC it is reasonable to consider 
a transparent case which highlights its relevance. For this the simplest setting appears 
to be a linear system of two, fully coupled, equations, such as $u_x=v$ and $v_y=u$, which
yields the  equation $u_{xy}=u$. In this case, the SCC puts a restriction on the 
relative location of the data manifolds $M$ and $N$ for $u$ and $v=u_x$, respectively. 

Given the extensive literature on second order hyperbolic PDEs (see Section 
\ref{review} for a 
selective review), our expectation was that, either the SCC was not actually
necessary, 
or it was ``well-known'' that a condition like the SCC is indeed required to 
establish well-posedness in this situation. 

However, we have not been able to find a treatment of this 
issue in the literature. In the present work we show that the SCC is indeed 
required for the uniqueness of a local solution by providing an example of 
non-uniqueness for the equation $u_{xy}=u$ when the SCC is not met. 
We then establish an existence result for more 
general (possibly nonlinear) second order hyperbolic equations.

In subsection \ref{issue} we explain in more detail the issue 
at hand. Section \ref{review} reviews related works. In Section \ref{exist_uxy=u} 
we establish existence via successive approximations for the model 
equation $u_{xy}=u$ in the case when the SCC fails. For the same 
case, we show in Section \ref{non_uniqueness} that uniqueness 
can fail. As the proof of existence in Section \ref{exist_uxy=u} involves a 
choice of additional data assignment along a certain curve 
(located away from the origin), 
the non-uniqueness result in Section \ref{non_uniqueness} may appear 
unsurprising. However, we show that the additional data assignment 
influences the solution all the way to the origin.
Finally, in Section \ref{local_existence} we apply the same technique
(additional data assignment) to show existence  of a solution to \eq{eqn} in the case when the SCC fails.
However, the additional data assignment must now be  done adaptively at each step in the 
iteration scheme.

\subsection{The issue}\label{issue} To highlight the issue, consider 
the following simple situation where both
data curves are straight lines through the origin: 
\[M:=\{x=a y\}\qquad N:=\{y=b x\},\]
where $a>0$ and $b>0$ are constants. We assume that $u$ and 
$u_x$ are given along parts of $M$ and $N$, respectively, in the first 
quadrant:
\beq\label{data1}
	u(ay,y)=\vp(y) \qquad 0\leq y\leq y_A,
\eeq
and
\beq\label{data2}
	u_x(x,bx)=\psi(x) \qquad 0\leq x\leq x_B.
\eeq

\begin{figure}
	\centering
	\includegraphics[width=13cm,height=6cm]{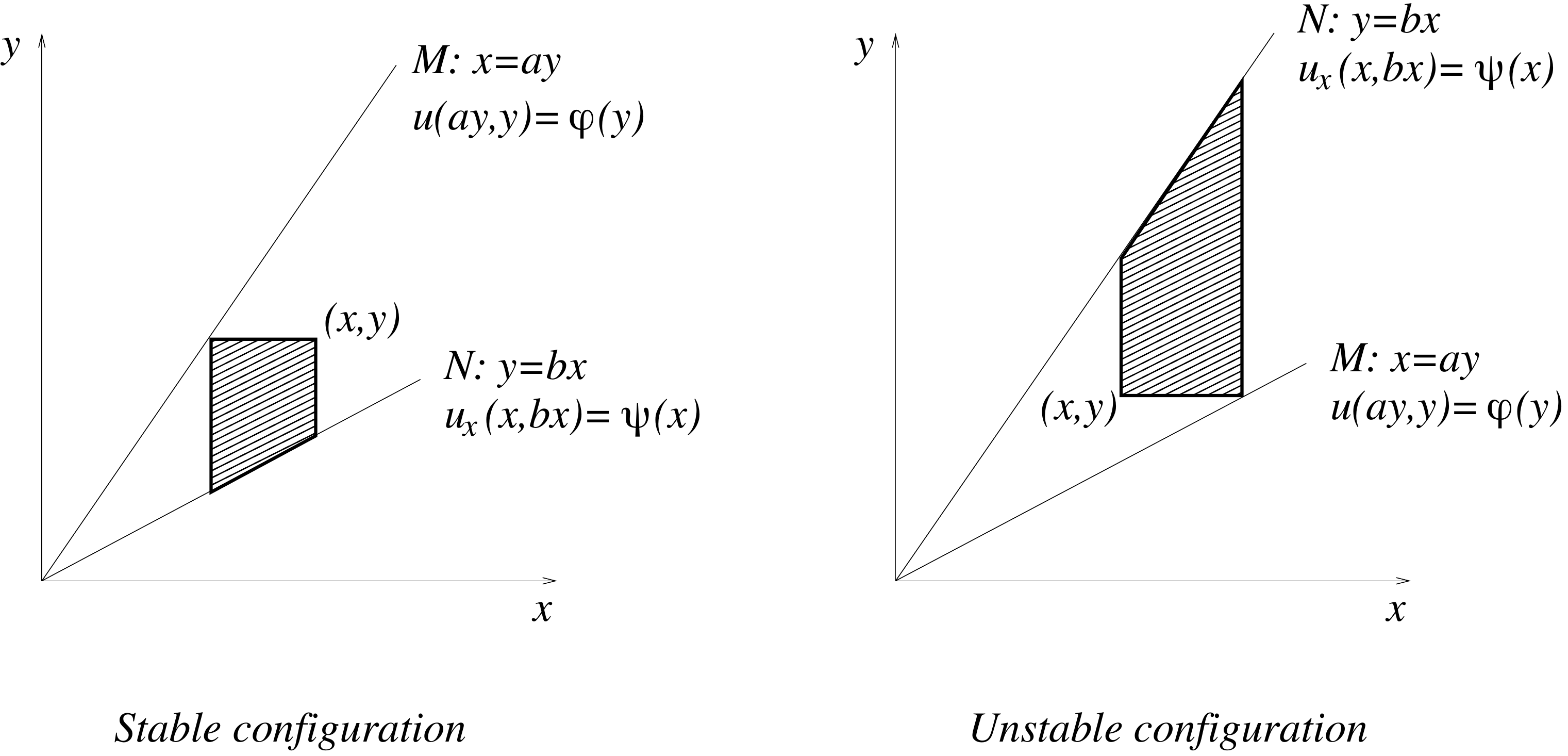}
	\caption{Stable and unstable configurations.}
	\label{Figure_cases}
\end{figure} 

\noindent 
Note that the data are only prescribed locally (the natural setting 
for nonlinear equations). 
Now consider the two cases,
\begin{itemize}
	\item[(I)] $ab<1$: the data curve $N$ for $u_x$ lies 
	below the data curve $M$ for $u$;
	\item[(II)] $ab>1$: vice versa;
\end{itemize}
see Figure~\ref{Figure_cases}.
Already for the elementary case where the right-hand side of \eq{eqn}
is independent of $(u,u_x,u_y)$, i.e., $f=f(x,y)$, there is a notable difference in
how the value of $u$ is determined at a point $(x,y)$ located between 
$M$ and $N$. In either case we first integrate with 
respect to $y$, exploiting the equation and the given $u_x$-data, 
followed by an $x$-integration to  exploit the $u$-data, obtaining:
\begin{align}
	\nn u(x,y)&=\vp(y)+\int_{ay}^x\left(\psi(\xi)\, 
	+\int_{b\xi}^yf(\xi,\eta)\, d\eta\right) d\xi\\
	&=\vp(y)+\int_{ay}^x\psi(\xi)\, d\xi
	+\int_{ay}^x\int_{b\xi}^yf(\xi,\eta)\, d\eta d\xi.\label{eq-itegration}
\end{align}
However, in Case (I),  the upper limits of integrations, given by  the 
coordinates of the point  $(x,y)$,  are larger than the  lower limits,  
and so all points in the  integration region  are closer to the origin than 
the point $(x,y)$ (shaded region in left part of Figure \ref{Figure_cases}).  
In contrast, for Case (II),  the upper limits of integrations, 
given by  the coordinates of the point  $(x,y)$,  are smaller than the  
lower limits. Thus we need to know $f$, $\phi$ and $\psi$ at points 
located \emph{further away from the origin than $(x,y)$}
in order to determine $u(x,y)$ (right part of Figure \ref{Figure_cases}).

This difference between the two cases is harmless  when  the right-hand side of \eq{eqn}
is independent of $(u,u_x,u_y)$: we obtain a 
(local) solution in either case. However, if we instead have a nontrivial
equation such as
\beq\label{lin}
	u_{xy}=u,
\eeq
a standard approach would apply Picard iteration where 
the $n$th iterate $u^{(n)}$ defines the next iterate $u^{(n+1)}$ from the equation
\[u_{xy}^{(n+1)}=u^{(n)}(x,y),\]
together with the original data requirements along $M$ and $N$. That is, 
$u^{(n)}(x,y)$ now plays the role of $f(x,y)$ above. It is clear that Case (II)
will, in a finite number of iterations, require data values that have not 
been assigned if the data are only prescribed locally. 

In contrast, it can be shown (see \cite{bjk}) that in Case (I) the iteration scheme 
outlined above will converge to a unique, local solution to \eq{lin}. We thus have a situation 
where the natural iteration scheme is well-defined and converges provided 
the data curves are located in a certain manner, while the same iteration 
scheme is undefined when this is not the case.
We express this by saying that the {\em Stable Configuration Condition} (SCC)
is satisfied in Case (I), while it is violated in Case (II).

As a remedy for Case (II) we shall fix a bounded set whose boundary 
consists of one part of each of $M$ and $N$, together with a new, chosen 
curve joining these, along which we prescribe values of $u$.
(These new values are subject to some mild compatibility conditions.)
For concreteness we choose the new curve to be a vertical line segment. 
In Section \ref{exist_uxy=u} we use this approach to establish local existence 
in Case (II) for the model equation $u_{xy}=u$ with general boundary data 
\eq{data1}-\eq{data2}.

However, we also establish non-uniqueness. More precisely, 
fix a triangle $\OAB$ as in Figure~\ref{Figure_uxy=u_data} below, and let $u$ and $u_x$ be prescribed 
to vanish along $\OA$ and $\OB$, respectively. 
Clearly, $u\equiv 0$ is a solution in this case. On the other hand, by 
prescribing a non-trivial function along the vertical segment $\AB$, we show
that there exist a non-trivial solution to the same problem 
which takes nonzero values at points arbitrarily close to the origin.
I.e., the values that we ``artificially'' assign to $u$ along $\AB$, propagate
inward and influence the solution arbitrarily close to the origin.

In Section \ref{local_existence}, we use a similar iteration scheme to  prove the existence of the solution  to a  general,
nonlinear equation of the form \eq{eqn} with a boundary data prescribed in the non-stable   configuration  case.
The crucial difference  with the linear case  treated in Section
\ref{exist_uxy=u},  is that additional  $u$-data along $\AB$ has to be chosen  adaptively as a part of the iteration scheme.
The local convergence  is of the scheme is  proven under Lipschitz conditions on $f$.

Before starting the analysis, we include a selective review of 
the substantial literature on boundary value problems for 
second order hyperbolic equations of the type \eq{eqn}.

\section{Review of related results}\label{review}
It is an understatement to say that the literature on hyperbolic 
second order PDEs of the form \eq{eqn} is extensive. Yet, to the
best of our knowledge, the particular boundary problem described
above has not been resolved when the relative location of the 
data curves is as in Case (II).

For linear equations, {\em viz.}
\beq\label{linear}
	u_{xy}=\alpha u_x+\beta u_y+\gamma u+\delta,
\eeq
where the coefficients $\alpha$-$\delta$ are functions of $(x,y)$,
the consideration of various types of boundary value problems 
dates, at least, back to Riemann's seminal work \cite{ri_coll}
where he introduced the Riemann function for a specific case.

In \cite{darb1} Darboux generalized and systemized Riemann's work,
and applied Picard iteration to solve the characteristic boundary 
value problem for \eq{linear}, i.e., the case when $u$ is prescribed along one 
vertical and one horizontal line. Part IV \cite{darb2} of the same work
contains a note by Picard outlining the method of successive 
approximations and its use for 2nd order hyperbolic equations. 
In this note, Picard treats the characteristic boundary value problem 
for both linear and nonlinear equations \eq{eqn}, and for the former 
he also considers the problem where $u$-data are 
prescribed along one characteristic curve as well along one 
non-characteristic curve. 

In Chapter XXVI of his Cours d'Analyse, 
Goursat \cite{gour} provided a detailed exposition of these results 
and also treated the Cauchy problem. In addition, Goursat 
considered the issue of integrating the elementary equation $u_{xy}(x,y)=F(x,y)$, 
where $F$ is a given function, with prescribed $u$-values
along two non-characteristic curves in the first quadrant.

The notes \cite{pic} by Picard again 
treat a series of boundary value problems for hyperbolic second order 
equations, both linear and nonlinear:
\begin{itemize}
	\item[(i)] the characteristic boundary value problem 
	($u$ prescribed along two characteristics, one of each family);
	\item[(ii)] the Cauchy problem ($u$ together 
	with one of its first partial derivatives prescribed along a single, 
	non-characteristic curve);
	\item[(iii)] $u$ prescribed along one characteristic and one 
	non-characteristic curve (both passing through a given point);
	\item[(iv)] $u$ prescribed along two non-characteristic curves 
	(with both passing through a given point; two cases are considered: 
	both curves lie in the first quadrant, or one lies in the first and 
	one lies in the fourth). 
\end{itemize}
For linear equations the existence of a unique solution 
is obtained in each of these ``classical'' cases via successive 
approximations, \cite{pic}.

While everybody agrees that problem (ii) above
be named the {\em Cauchy problem}, the terminology for the other 
problems is not uniform in the literature, \cite{nak}. E.g., 
Pogorzelski \cite{pog} refers to problems (i) and (iii) 
above as the {\em Darboux problem} and the {\em Picard problem},
respectively. Walter \cite{wal} and Kharibegashvili \cite{khar} refer 
to (i) as the {\em Goursat problem}, while Lieberstein \cite{lieb} uses this 
for problem (iv). Other authors \cites{ch,gar,dib} call (i) the {\em 
characteristic initial value (or characteristic Goursat) problem}. 
Kharibegashvili \cite{khar} also refers to 
problems (iii) and (iv) as the {\em 1st and 2nd Darboux problems}, 
respectively. 

In any case, none of the works mentioned thus far addresses the 
boundary value problem \eq{eqn}-\eq{data1}-\eq{data2} where $u$ 
is prescribed along one non-characteristic curve and $u_x$ is 
prescribed along another non-characteristic curve (with both curves located 
in the first quadrant). 

The 1940s and 50s saw renewed interest in hyperbolic equations. 
In particular, various new types of boundary value problems were
considered for equations of the form \eq{eqn}.
Among the many works in this area we shall only comment on those 
that are most closely related to the problem described in Section \ref{intro}.
(For an extensive  bibliography, see Walter \cite{wal}.) We shall 
see that, while the results in these works do apply to Case (I) above,
none of them covers Case (II). The reason is essentially the same in each
case: the various setups put a restriction on the relative location of 
the data curves for $u$ and $u_x$, excluding Case (II).

Szmydt \cites{szm1,szm2} considered two generalized boundary value problems 
for equations of the form \eq{eqn}. For both types of problems, the 
data are prescribed along two curves, $\Gamma=\{y=\gamma(x)\}$ defined for 
$-\alpha\leq x\leq \alpha$, and $\Lambda:=\{x=\lambda(y)\}$ defined for 
$-\beta\leq y\leq \beta$. It is explicitly assumed that $\alpha$, $\beta$ are finite, 
and that $\Gamma$ and $\Lambda$ are situated within the rectangle 
$D=[-\alpha,\alpha]\times [-\beta,\beta]$.
For the first type of problem Szmydt prescribes $u(x_0,y_0)=u_0$ at an arbitrary 
point $(x_0,y_0)\in D$,
\beq\label{szm1}
	u_x(x,y)=G(x,u(x,y),u_y(x,y))\qquad\text{along $y=\gamma(x)$,}
\eeq
and similarly $u_y$ as a function $H(y,u,u_x)$ along $x=\lambda(y)$. 
For the second type of problem, $u_x$ is again 
prescribed according to \eq{szm1}, while it is required that
\beq\label{szm2}
	u(\lambda(y),y)=u_0+\int_{y_0}^y B(t,u(\lambda(t),t),u_x(\lambda(t),t))\, dt
\eeq
holds for $-\beta\leq y\leq \beta$. Here, $u_0$ is a constant, $-\beta\leq y_0\leq \beta$,
and $B$ is a continuous function. As Szmydt  points out in Remark 1 of \cite{szm1}, 
if $\lambda$ is of the class $C^1$, then the two problems are equivalent. 
Under various conditions, existence 
of a local solution is obtained for each type of problem. 
As pointed out in \cite{szm2}, the results 
cover the classical problems (i)-(iii) listed above. 

Now consider the linear equation \eq{eqn} with data prescribed 
along two straight lines $M$ and $N$ as in  
\eq{data1}-\eq{data2}, where we assume  $ab>1$ (Case (II)). 
To formulate this in the setup of Szmydt \cites{szm1,szm2} we must 
let $\Gamma=M$ and $\Lambda=N$, since the curves are to be contained within
the rectangle $D$. However, in either problem considered by Szmydt,  
$u_x$ is prescribed along $\Gamma=M$, while in our assignment \eq{data1},
$u$ is assigned along $M$. Thus, Szmydt's setup does not apply 
to Case (II). 

Next, the concise work \cite{cil} by Ciliberto considers the following boundary 
value problem for \eq{eqn} (``Problem (A)'' in \cite{cil}).
We are given two curves 
\[C_1=\{y=\alpha(x)\,:\, a\leq x\leq b\} \qquad \text{and}\qquad
C_2=\{x=\beta(y)\,:\, c\leq y\leq d\},\]
where 
\beq\label{ranges}
	c\leq\alpha(x)\leq d \quad\text{for $a\leq x\leq b$, and}\quad
	a\leq\beta(y)\leq b\quad\text{for $c\leq y\leq d$,}
\eeq
and $\alpha(0)=\beta(0)=0$. Then: determine a solution of \eq{eqn} 
with assigned values of $u$ along $C_2$ and with assigned values of
$u_x$ along $C_1$. Under suitable conditions, Ciliberto establishes the 
existence of a (global) solution to this problem. (There is considerable 
overlap between Ciliberto's and Szmydt's results; see footnote 
(${}^3$) on p.\ 384 in \cite{cil}.)

If we try to apply this setup to \eq{eqn} with data prescribed 
along two straight lines $M$ and $N$ as in  
\eq{data1}-\eq{data2}, we encounter the same issue as with Szmydt's 
setup. Namely, since, according to \eq{ranges}, the curves should be 
situated within the rectangle $[a,b]\times[c,d]$, $C_1$ must be chosen 
as $M$ and $C_2$ must be chosen as $N$. However, this is precisely 
Case (I) in Section \ref{intro}, showing that the analysis in \cite{cil} does 
not apply to Case (II). 

Ciliberto remarks (footnote (${}^2$) on p.\ 384 of \cite{cil}) that an entirely
equivalent problem is obtained if $u$ is prescribed along $C_1$ and $u_y$ 
is prescribed along $C_2$. This being the case, it is surprising to us that no remark is made about 
the (non-equivalent) problem where $u$ is prescribed along $C_1$ and $u_x$ 
is prescribed along $C_2$ (i.e., Case (II) above).

Next we consider the detailed work \cite{azdz} by Aziz \& Diaz 
on linear equations of the form
\beq\label{az1}
 	u_{xy}+a(x,y)u_x+b(x,y)u_y+c(x,y)=d(x,y),
\eeq
and with boundary conditions of the form
\beq\label{bc1}
	\alpha_0(x)u+\alpha_1(x)u_x+\alpha_2(x)u_y=\sigma(x) \qquad\text{on $y=f_1(x)$, $x\in[0,x_0]$,}
\eeq
and
\beq\label{bc2}
	\beta_0(y)u+\beta_1(y)u_x+\beta_2(y)u_y=\tau(y) \qquad\text{on $x=f_2(y)$, $y\in[0,y_0]$,}
\eeq
and
\beq\label{bc3}
	u(0,0)=\gamma,
\eeq
for given functions $\alpha_i$, $\beta_i$, $f_i$, and a given constant $\gamma$.
The problem \eq{az1}-\eq{bc1}-\eq{bc2}-\eq{bc3} is posed on a characteristic 
rectangle $R:=[0,x_0]\times[0,y_0]$, and it is explicitly assumed in \cite{azdz}
that 
\beq\label{azdz_assumpn1}
	0\leq f_1(x)\leq y_0\qquad\text{with $f_1(x)$ defined for all $x\in[0,x_0]$,}
\eeq
and that
\beq\label{azdz_assumpn2}
	0\leq f_2(y)\leq y_0\qquad\text{with $f_2(y)$ defined for all $y\in[0,y_0]$.}
\eeq
Three cases are treated by Aziz \& Diaz: (here, e.g.\ ``$\alpha_1,\, \beta_2\neq 0$'' means that 
$\alpha_1(x)\neq 0$ for all $x\in[0,x_0]$ and $\beta_2(y)\neq 0$ for all $y\in[0,y_0]$)
\begin{itemize}
	\item[(*)] $\alpha_1,\, \beta_2\neq 0$ and $\alpha_2=\beta_1\equiv 0$.
	\item[(**)] $\alpha_1,\, \beta_2\neq 0$, but with additional conditions imposed on $f_i$, $\sigma$ and $\tau$.
	\item[(***)] $\alpha_0,\, \beta_0\neq0$, $\alpha_1=\alpha_2=\beta_1=\beta_2\equiv0$.
\end{itemize}
Under suitable assumptions, the authors establish 
existence and uniqueness of a solution to \eq{az1}-\eq{bc1}-\eq{bc2}-\eq{bc3} 
in each case. 

To apply this setup in Case (II) of the problem \eq{eqn}-\eq{data1}-\eq{data2} in Section \ref{intro}
we would need to differentiate \eq{data1} with respect to $y$, 
and then see if (**) can be applied. This, however, fails for the 
following reason. The differentiated condition is
\[au_x(ay,y)+u_y(ay,y)=\vp'(y),\qquad\text{(together with $u(0,0)=\vp(0)$),}\]
and this must play the role of \eq{bc2} above (because the other boundary condition is
$u_x(x,bx)=\psi(x)$, which does not contain $u_y$). It follows that the curves $\Gamma_1$ and $\Gamma_2$ 
in \cite{azdz} must correspond to the curves $N$ and $M$, respectively,
in our Case (II). Thus, $\Gamma_1$ lies above $\Gamma_2$ in the first quadrant.
But then it is not possible to find $x_0$ and $y_0$ so that both \eq{azdz_assumpn1} and
\eq{azdz_assumpn2} are satisfied. 
(We note that the work \cite{azdz} provides a very detailed review and criticism of 
earlier works.)

We finally mention the more recent work \cite{kj} which treats 
nonlinear equations of the type \eq{eqn} in angular domains. Again, 
a careful reading reveals that the analysis in \cite{kj}
applies to Case (I), but not to Case (II).

\section{Existence for $u_{xy}=u$ in Case (II)}\label{exist_uxy=u}
In this section, we establish existence of a local solution to the 
following mixed boundary value problem:
\beq\label{uxy=u}
	u_{xy}=u
\eeq
with data
\beq\label{data3}
	u(ay,y)=\vp(y) \qquad 0\leq y\leq y_A,
\eeq
and
\beq\label{data4}
	u_x(x,bx)=\psi(x) \qquad 0\leq x\leq x_B,
\eeq
where $\vp$ is a $C^1$ function on $[0, y_A]$ and $\psi$ is a $C^0$ function on $[0,x_A]$.
We assume $ab>1$, so that we are in Case (II) (see Figure~\ref{Figure_uxy=u_data}),  and 
we proceed to write down an iteration scheme which will provide 
existence of a local solution. 

\begin{figure}
	\centering
	\includegraphics[width=10cm,height=8cm]{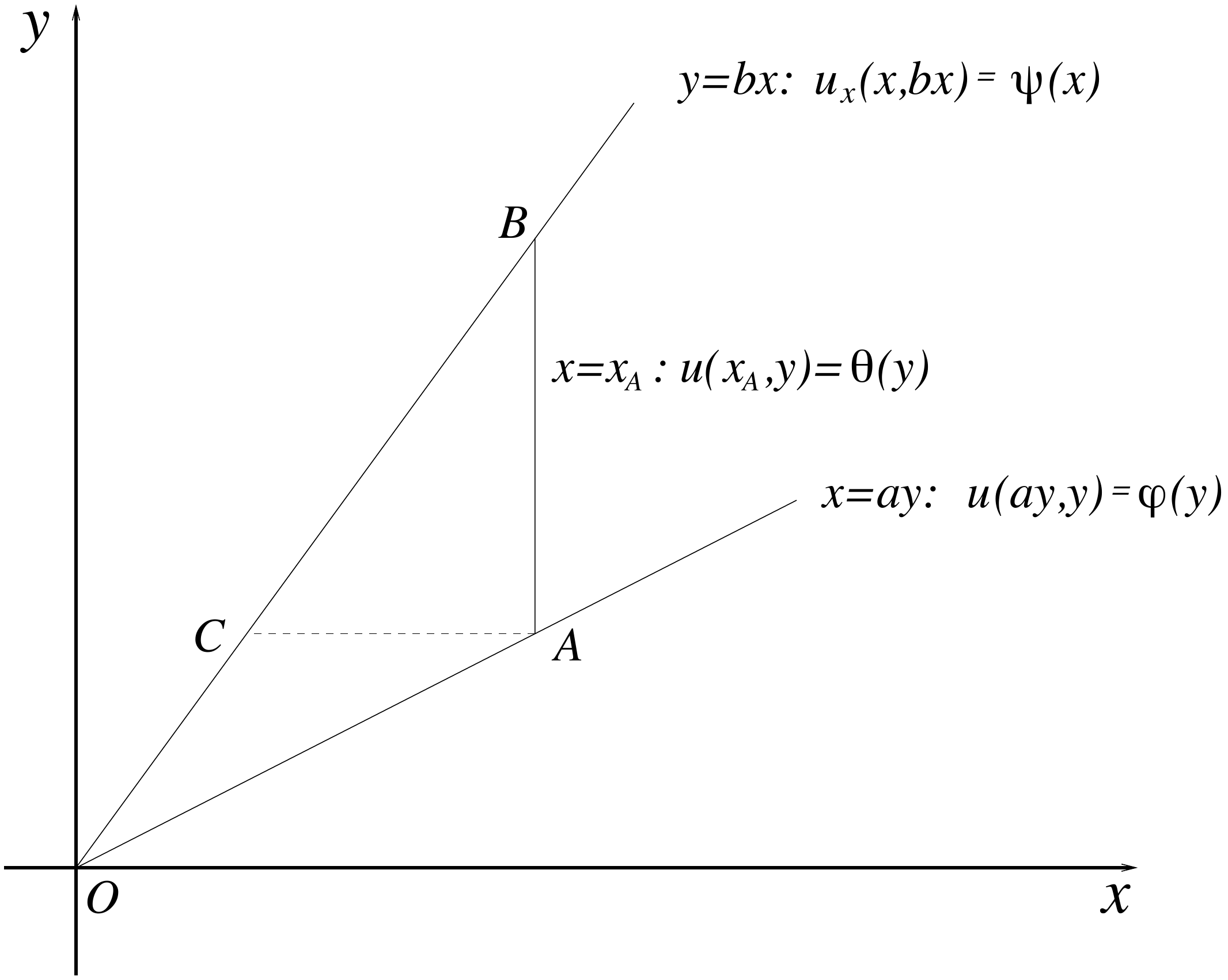}
	\caption{$u_{xy}=u$: in addition to the original data $\vp$ and $\psi$,
	we also prescribe $u$-data $u(x_A,y)=\theta(y)$ along the vertical line $AB$.}
	\label{Figure_uxy=u_data}
\end{figure} 

As outlined above, we circumvent the problem in Case (II) (i.e., requiring data 
progressively farther away from the origin for the iterates), by 
artificially prescribing $u$-data along the vertical segment $\AB$; see Figure~\ref{Figure_uxy=u_data}.
This results in a ``split'' scheme with different expressions in the triangles 
$\OAC$ and $\ABC$, where $C=(\frac{y_A}{b},y_A)$. The $u$-data along 
$\AB$ are given by
\beq\label{data5}
	u(x_A,y)=\theta(y)\qquad y_A\leq y\leq y_B,
\eeq
where $\theta$ remains to be specified. It must be chosen so that 
the iterates remain continuous on the full closed triangle $\overline{OAB}$. 


\subsection{Iteration scheme}
The iteration scheme is obtained by having $u^{(n+1)}$ be the 
solution of $u_{xy}(x,y)=u^{(n)}(x,y)$ in $\overline{OAB}$, with the originally assigned
boundary data along $\OA$ and $\OB$, 
together with $u$-data assigned according to \eq{data5} 
along $\AB$. However, continuity of the solution throughout the triangle 
$\overline{OAB}$ (in particular, across the horizontal segment $\AC$)
puts restrictions on $\theta$. To determine these we first consider 
the elementary case of $u_{xy}=f(x,y)$. The solution  
in this case is given by the split expression
\beq\label{OAC_soln}
	u(x,y)=
	\vp(y)-\int_x^{ay}\psi(\xi)\, d\xi+\int_x^{ay}\int_y^{b\xi}f(\xi,\eta)\, d\eta d\xi
	\qquad\text{on $\overline{OAC}$,}
\eeq
and
\beq\label{ABC_soln}
	u(x,y)=\theta(y)-\int_x^{x_A} \psi(\xi)\, d\xi+\int_x^{x_A} \int_y^{b\xi}f(\xi,\eta)\, d\eta d\xi
	\qquad\text{on $\overline{ABC}$.}
\eeq
A first issue is what restrictions need to be imposed on $\theta$ to 
guarantee that the function $u$ given by \eq{OAC_soln}-\eq{ABC_soln} 
is consistently defined on the segment $\AC$ and is of the class $C^1$.
This is answered by:
\begin{lemma}\label{regularity}
        Consider the mixed boundary value problem for 
        the equation $u_{xy}=f(x,y)$ with data 
        \eq{data3}-\eq{data4}-\eq{data5}, where $ab>1$ and
        \begin{itemize}
        	\item[(A1)] $\vp\in C^1[0,y_A]$
        	\item[(A2)] $\psi\in C^0[0,x_A]$
        	\item[(A3)] $f\in C^0(\overline{OAB})$.
        \end{itemize}
        Let $\theta:[y_A,y_B]\to\RR$ be any function satisfying
        \begin{itemize}
        	\item[(A4)] $\theta\in C^1[y_A,y_B]$
        	\item[(A5)] $\theta(y_A)=\vp(y_A)$
        	\item[(A6)] $\displaystyle{\theta'(y_A)=\vp'(y_A)-a\psi(x_A)
        	+a\int_{y_A}^{y_B} f(x_A,\eta)\, d\eta.}$
        \end{itemize}
        Then the piecewise defined function $u(x,y)$ given by 
        \eq{OAC_soln}-\eq{ABC_soln} provides a solution
        in the following sense: $u$, $u_x$, $u_y$, and $u_{xy}=u_{yx}$ are 
        continuous on $\overline{OAB}$, $u_{xy}(x,y)=f(x,y)$ on $\overline{OAB}$, 
        and $u$ takes the values \eq{data3}-\eq{data4}-\eq{data5} 
        on the boundary of $\overline{OAB}$.
\end{lemma}
\begin{proof}
Continuity of $u$ in the  triangles $\OAC$ and $\ABC$ is  
clear from \eq{OAC_soln}-\eq{ABC_soln}, combined with (A1)-(A4).
 To show that  \eq{OAC_soln}-\eq{ABC_soln} consistently define $u$ along the common side $ \AC$, we substitute
\beq\label{AC} y=y_A=\frac{x_A}{a}\eeq
in  \eq{OAC_soln}-\eq{ABC_soln},  respectively. We get for $(x,y_A)\in \AC$:
\begin{align}\label{OAC_AC}
	u(x,y_A)&=
	\vp(y_A)-\int_x^{ay_A}\psi(\xi)\, d\xi
	+\int_x^{ay_A}\int_{y_A}^{b\xi}f(\xi,\eta)\, d\eta d\xi, \\
	\nonumber &=\vp(y_A)-\int_x^{x_A}\psi(\xi)\, d\xi
	+\int_x^{x_A}\int_{y_A}^{b\xi}f(\xi,\eta)\, d\eta d\xi, \\
	\label{ABC_AC}
	u(x,y_A)&=
	\theta(y_A)-\int_x^{x_A} \psi(\xi)\, d\xi
	+\int_x^{x_A} \int_{y_A}^{b\xi}f(\xi,\eta)\, d\eta d\xi.
\end{align}
 We see that (A5) is the only condition that we need  in order to  guarantee 
 that   $u$ is consistently  defined along $\AC$ and is continuous.
 
  To check that $u$ takes the assigned  
values on $\OA$, we substitute $x=ay$  into  \eq{OAC_soln}:
\begin{align}\label{OAC_x=ay}
	u(ay,y)&=\vp(y)-\int_{ay}^{ay}\psi(\xi)\, d\xi
	+\int_{ay}^{ay}\int_y^{b\xi}f(\xi,\eta)\, d\eta d\xi=\vp(y).
\end{align}
Differentiation  of \eq{OAC_soln} and \eq{ABC_soln} with respect 
to $x$ produces the same formula:
\beq\label{ux} u_x=\psi(x)-\int_y^{bx}f(x,\eta)d\eta,\eeq
which is clearly continuous on the entire region $\overline{OAB}$  and takes 
the assigned values $\psi(x)$ along $\OB$, where $y=bx$. 
Differentiation  of \eq{OAC_soln} and \eq{ABC_soln} with 
respect to $y$ produces two different formulas:
\begin{align}
	\label{OAC_uy}
	u_y(x,y)&=
	\vp'(y)-a\psi(ay)+a\int_y^{aby}f(ay,\eta)\, d\eta -\int_x^{ay}f(\xi,y)\, d\xi
	\qquad\text{on $\OAC$},\\
	\label{ABC_uy}
	u_y(x,y)&=\theta'(y)-\int_x^{x_A}f(\xi,y) d\xi
	\qquad\text{on $\ABC$.}
\end{align}
Continuity of $u_y$ on  $\OAC$ and $\ABC$ 
is clear from \eq{OAC_uy}-\eq{ABC_uy}, combined with (A1)-(A4).
To show that  \eq{OAC_uy}-\eq{ABC_uy} consistently define $u$ along the common side $ \AC$, we substitute \eqref{AC} 
into these formulas:
\begin{align}\nonumber
	u_y(x,y_A)&=
	\vp'(y_A)-a\psi(ay_A)+a\int_{y_A}^{aby_A}f(ay,\eta)\, d\eta 
	-\int_x^{ay_A}f(\xi,y)\, d\xi\\
	\label{OAC_uy-AC}&=\vp'(y_A)-a\psi(x_A)
	+a\int_{y_A}^{y_B}f(ay,\eta)\, d\eta -\int_x^{x_A}f(\xi,y)\, d\xi\\
	\label{ABC_uy-AC}
	u_y(x,y_A)&=\theta'(y_A)-\int_x^{x_A}f(\xi,y) d\xi.
\end{align}
We see that (A6) is the only condition needed  in order to guarantee 
that $u_y$ is is consistently  defined along $\AC$ and is continuous.

Finally, by differentiating  \eqref{ux} with respect to $y$ and 
differentiating  \eqref{OAC_uy}  and \eqref{ABC_uy} with 
respect to $x$ we immediately  see that
\[u_{xy}=u_{yx}=f(x,y).\]
\end{proof}

We now return to the equation $u_{xy}=u$ with data \eq{data3}-\eq{data4}-\eq{data5}.
Let the $n$th iterate $u^{(n)}(x,y)$ play the role of 
$f$ above. Substituting $\theta(y)$ for $u^{(n)}(x_A,y)$ in 
(A6), we now fix any function $\theta\in C^1[y_A,y_B]$ satisfying
\beq\label{theta_conds}
	\theta(y_A)=\vp(y_A)\qquad\text{and}\qquad
	\theta'(y_A)=\vp'(y_A)-a\psi(x_A)
        	+a\int_{y_A}^{y_B} \theta(\eta)\, d\eta.
\eeq
It is clear that there are many possible choices for $\theta$.   E.g., we 
may choose $\theta(y)$ as an affine function or quadratic polynomial 
in $y$.  A straightforward calculation shows that in either case, 
$|\theta(y)|$ may be bounded by a constant depending on $(y_B-y_A)$ and upper 
bounds on $|\psi|$, $|\vp|$, and $|\vp'|$.
We can now specify a suitable iteration scheme. We start by setting
\beq\label{1st_step_OAC}
	u^{(0)}(x,y):=\vp(y)-\int_x^{ay}\psi(\xi)\, d\xi \qquad \text{on $\overline{OAC}$}
\eeq
and
\beq\label{1st_step_ABC}
	u^{(0)}(x,y):=\theta(y)-\int_x^{x_A}\psi(\xi)\, d\xi \qquad \text{on $\overline{ABC}.$}
\eeq
Note that $u^{(0)}$ is well defined on the overlap $\AC$ of the two regions, 
thanks to \eq{theta_conds}${}_1$, and so is continuous throughout $\overline{OAB}$.  In addition, $u^{(0)}$ equals  to $\theta$ 
along $\AB$.

The first iterate $u^{(1)}(x,y)$ is then defined according to 
\eq{OAC_soln}-\eq{ABC_soln} with $u^{(0)}$ substituted for $f$. 
By definition, $u^{(1)}$ reduces to $\theta$ 
along $\AB$, and, due to the properties \eq{theta_conds} of $\theta$, 
$u^{(1)}$ is continuous throughout $\overline{OAB}$ 
according to Lemma \ref{regularity}.

Next, assuming the $n$th iterate $u^{(n)}$ has been determined, and is 
continuous, throughout $\overline{OAB}$, we define the next iterate $u^{(n+1)}$ 
according to \eq{OAC_soln}-\eq{ABC_soln} with $u^{(n)}$ substituted for $f$. 
Again, by definition $u^{(n+1)}$ reduces to $\theta$ 
along $\AB$, and, due to the properties \eq{theta_conds} of $\theta$, 
$u^{(n+1)}$ is continuous throughout $\overline{OAB}$ 
according to Lemma \ref{regularity}.

\subsection{Convergence}
According to \eq{1st_step_OAC}-\eq{1st_step_ABC} we have
\[|u^{(0)}(x,y)|\leq \C\qquad\text{for $(x,y)\in \overline{OAB}$,}\]
for a finite constant $\C$ depending on the size of the triangle $\overline{OAB}$
and upper bounds on $|\psi|$, $|\vp|$, and $|\vp'|$. 
(Recall that $\theta$ 
may be chosen so as to be similarly bounded.) 
It follows that 
\[|u^{(1)}(x,y)-u^{(0)}(x,y)|\leq \C(x_A-x)(y_B-y)\qquad\text{for $(x,y)\in \overline{OAB}$,}\]
and that in general
\[|u^{(n)}(x,y)-u^{(n-1)}(x,y)|\leq \frac{\C}{(n!)^2}(x_A-x)^n(y_B-y)^n
\qquad\text{for $(x,y)\in \overline{OAB}$.}\]
It follows that the sequence $(u^{(n)})$ is {uniformly} Cauchy in 
$C^0(\overline{OAB})$, and thus converges uniformly to a 
limit function $u\in C^0(\overline{OAB})$. Recalling that 
\[u^{(n+1)}(x,y)=\vp(y)-\int_x^{ay}\psi(\xi)\, d\xi+\int_x^{ay}\int_y^{b\xi}u^{(n)}(\xi,\eta)\, d\eta d\xi
	\qquad\text{on $\overline{OAC}$,}\]
and
\[u^{(n+1)}(x,y)=\theta(y)-\int_x^{x_A} \psi(\xi)\, d\xi+\int_x^{x_A} \int_y^{b\xi}u^{(n)}(\xi,\eta)\, d\eta d\xi
	\qquad\text{on $\ABC$,}\]
and passing to the limit $n\to \infty$, shows that $u$ satisfies
\[u(x,y)=\vp(y)-\int_x^{ay}\psi(\xi)\, d\xi+\int_x^{ay}\int_y^{b\xi}u(\xi,\eta)\, d\eta d\xi
	\qquad\text{on $\OAC$,}\]
and
\[u(x,y)=\theta(y)-\int_x^{x_A} \psi(\xi)\, d\xi+\int_x^{x_A} \int_y^{b\xi}u(\xi,\eta)\, d\eta d\xi
	\qquad\text{on $\ABC$.}\]
Direct calculations then show that $u$ is a solution to the mixed 
boundary value problem \eq{uxy=u}-\eq{data3}-\eq{data4}.

\begin{remark}
	Unsurprisingly, for the linear problem under consideration we obtain 
	a global solution defined on all of $\OAB$. 
	Also, the occurrence in the calculations above of the power series 
	\[\sum_{n=0}^\infty \frac{z^{n}}{(n!)^2}=J_0(2i\sqrt{z}),\] 
	where $J_0$ denotes the Bessel function of the 1st kind of order zero,
	is natural: it is known that 
	the Riemann function for the linear equation $u_{xy}=cu$ ($c$ constant)
	can be expressed in terms of  $J_0$; e.g., 
	see Section 4.4 in \cite{gar}. 
\end{remark}

\section{Non-uniqueness for $u_{xy}=u$ in Case (II)}\label{non_uniqueness}
We proceed to show that the solution to \eq{uxy=u} found in the previous 
section is, in general, not uniquely determined by prescribing $u$ locally 
along $\{x=ay\}$ and $u_x$ locally along $\{y=bx\}$, when $ab>1$. 

Given the existence result in the previous section, and the freedom 
we have in assigning the function $\theta$ along $\AB$, 
it is immediate that different solutions can be obtained whose 
$u$-values agree along the  segment $\OA$ of the line  
$\{x=ay\}$ and whose $u_x$-values agree along  the segment 
$\OB$ of the  line $\{y=bx\}$ (see Figure~\ref{Figure_uxy=u_data}). Indeed, 
by continuity of the solutions on $\overline{OAB}$, it is 
clear that different choices of $\theta$ will give solutions that 
are distinct in some neighborhood of the segment $\AB$. 

However, we shall establish a stronger version of non-uniqueness,
showing that changes in $\theta$-values along $\AB$ can ``propagate''
into the domain $OAB$ and influence
the resulting solution at points arbitrarily close to the origin. 

For this the simplest 
thing appears to be to assign vanishing data for both $u$ and $u_x$, such that 
$u_0(x,y)\equiv 0$ is one solution. The issue then is to show that there 
is another solution $\hat u(x,y)$ which is not identically zero. This will be done
by carefully choosing a non-negative function $\theta$, satisfying the constraints
\eq{theta_conds}, i.e.,
\beq\label{0_theta_conds}
	\theta(y_A)=0\qquad\text{and}\qquad
	\theta'(y_A)=a\int_{y_A}^{y_B} \theta(\eta)\, d\eta,
\eeq
and then run the iteration scheme described in the previous section. 
We shall show that this yields a solution $\hat u$ which takes on strictly positive 
values arbitrarily close to the origin.

\begin{theorem}\label{main_result}
	Consider the boundary value problem \eq{uxy=u}-\eq{data3}-\eq{data4} 
	with vanishing boundary data $\vp\equiv 0$ and $\psi\equiv 0$. 
	For a given $\theta\in C^1[y_A,y_B]$ satisfying \eq{0_theta_conds}, 
	let $u_\theta$ denote the solution generated by the scheme in 
	Section \ref{exist_uxy=u}. Then: It is possible to choose $\theta$ so that 
	$u_\theta(x,y)>0$ at all points $(x,y)$ in the open  triangle $OAC$.
\end{theorem}

\begin{proof}
\begin{figure}
	\centering
	\includegraphics[width=10cm,height=10cm]{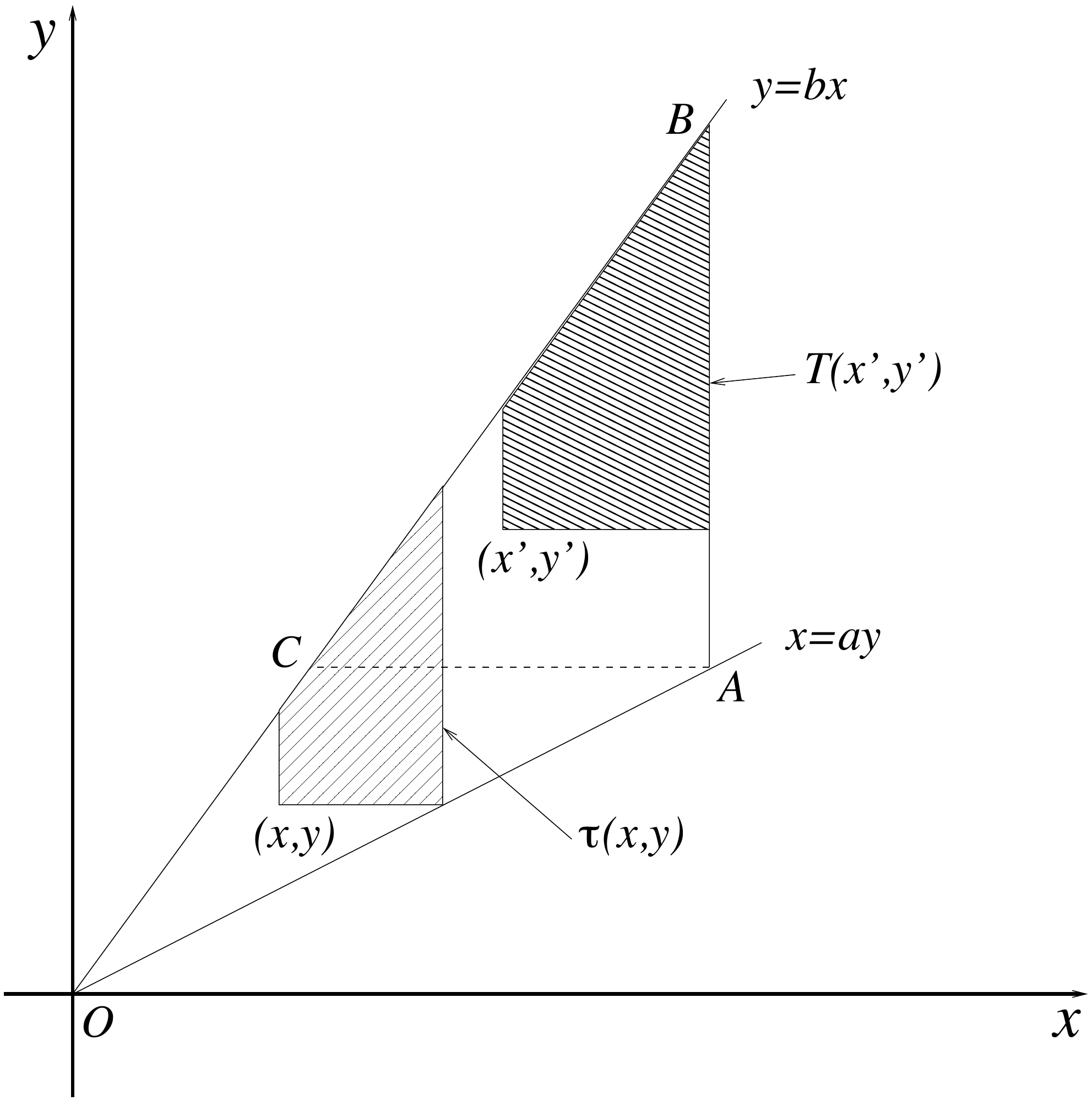}
	\caption{$u_{xy}=u$: the trapezoids $\tau(x,y)$ and $T(x',y')$.}
	\label{figure-uxy=u_trapezoids}
\end{figure} 
\begin{figure}
	\centering
	\includegraphics[width=10cm,height=10cm]{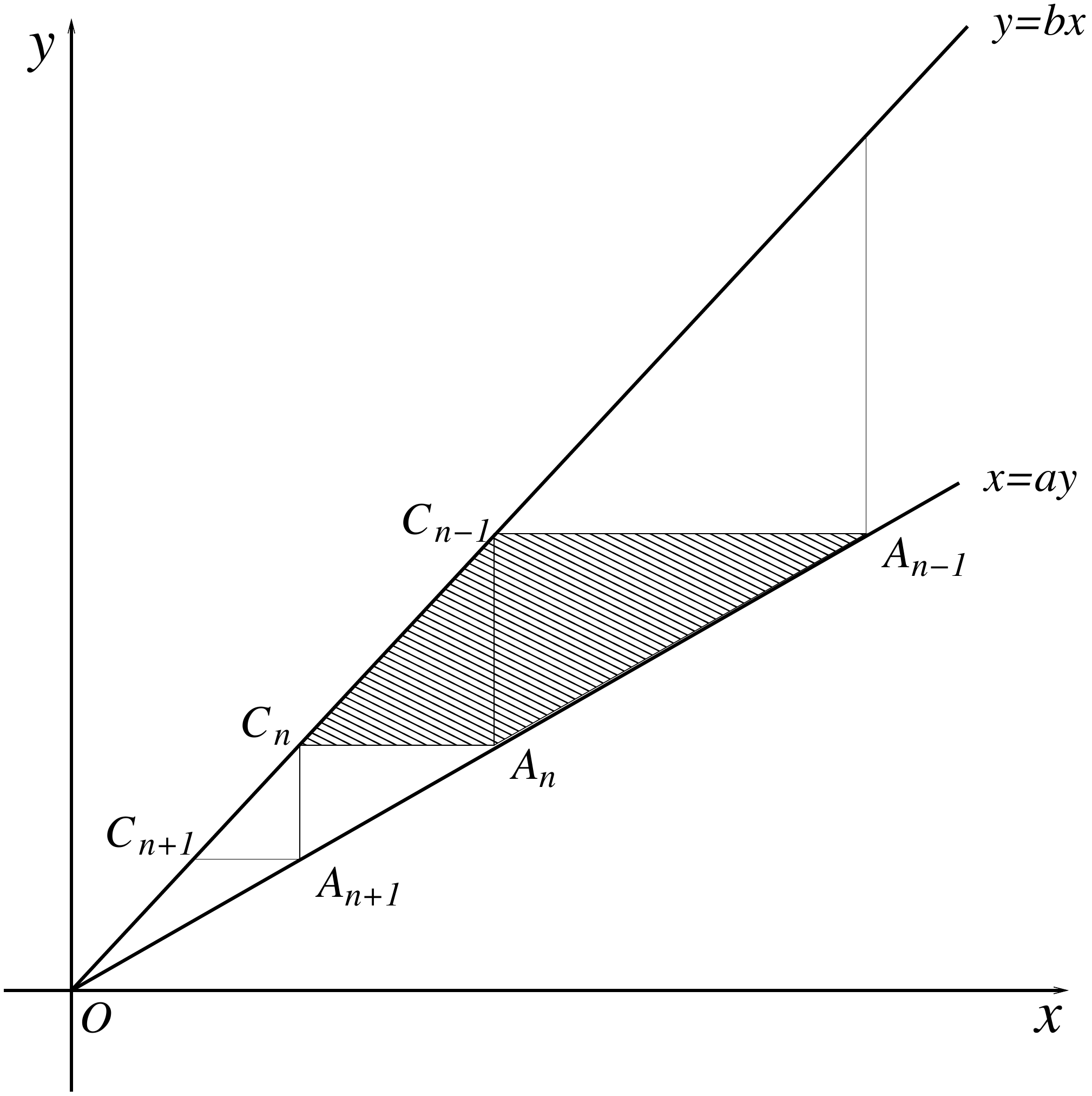}
	\caption{$u_{xy}=u$: the subdomains used in the proof of 
	Theorem \ref{main_result}.}
	\label{figure-uxy=u_points}
\end{figure} 
It is convenient to introduce the following notation (see Figure~\ref{figure-uxy=u_trapezoids}). 
For $(x,y)\in \ABC$, let  $T(x,y)$ be the closed trapezoid with vertices $(x,y)$, $(x_A,y)$, $B$,
and $(x,bx)$. For $(x,y)\in \OAC$, let $\tau(x,y)$  be the closed. trapezoid with vertices 
$(x,y)$, $(ay,y)$, $(ay,bay)$, and $(x,bx)$.  Note that  for $(x,y)\in \AC$, these trapezoids coincide. 
With $\theta$ as indicated, the iteration scheme in 
Section \ref{exist_uxy=u} for $u_{xy}=u$ with vanishing boundary 
data $\vp\equiv 0$ and $\psi\equiv 0$, is then:
\begin{itemize}
	\item for $(x,y)\in \OAC$:
	\beq\label{OAC_soln-0}
		\iu 0(x,y)=0  \quad \text{and} \quad 
		\iu{n+1}(x,y)= \iint_{\tau(x,y)}\!\!\iu n(\xi,\eta)\, d\eta d\xi;
	\eeq
	\item for $(x,y)\in \ABC$:
	\beq\label{ABC_soln-0}
		\iu 0(x,y)=\theta(y)  \quad \text{and} \quad 	
		\iu{n+1}(x,y)=\theta(y)+\iint_{T(x,y)}\!\!\iu n(\xi,\eta)\, d\eta d\xi.
	\eeq
\end{itemize}
We impose the additional requirement that $\theta$ satisfies $\theta(y)>0$ 
on the open interval $(y_A, y_B)$. It is not difficult to see that there are such functions 
which also satisfy \eq{0_theta_conds}.

The goal is to show that the strict positivity of $\theta$ ``propagates'' into the solution 
on all of $OAC$ in the  precise sense of Claim \ref{main_claim} below. First, we introduce 
the following notation. Set 
\begin{align*}
	A_0&:=A,\quad  C_0:=C,\\
	A_i:&=\big(x_{C_{i-1}}, \textstyle\frac{1}{a}x_{C_{i-1}}\big),\quad C_i:=\big( \textstyle\frac{1}{b} y_{A_i}, y_{A_i}\big)\quad\text{for $i=1,2,\dots$,}
\end{align*}
(see Figure~\ref{figure-uxy=u_points}), and define the trapezoidal regions: 
\[\mathcal T_N:=A_N A_{N-1}C_{N-1} C_N\cup {C_{N-1}A_{N-1}},\]
 with the top side included and other three sides excluded, where $C_{N-1}A_{N-1}$ is an open segment.

\begin{claim}\label{main_claim}
	For any $N\geq 1$, the following holds: for all $n\geq N$,
	$\iu {n}>\iu {n-1}$ on $\mathcal T_N$.
\end{claim}
Since the iterates $\iu n$ converge uniformly on $OAB$ to the limit solution $u_\theta$,
and since the trapezoids $\mathcal T_N$ in Claim \ref{main_claim} exhaust $OAC$ 
as $N\to\infty$,  Claim \ref{main_claim} implies that $u_\theta(x,y)>0$ at all points $(x,y)\in OAC$, 
which is the conclusion of the theorem. 

It thus remains to establish Claim \ref{main_claim}. This is carried out via double 
induction on $N$ and $n$. 
We first establish two auxiliary results.
\begin{claim}\label{sub_claim_1}
	For all $n\geq 0$ we have,
	\beq\label{nonneg}
		\iu n\geq 0 \qquad\text{on the triangle $\OAB$,}
	\eeq
	and
	\beq\label{eq-claim1}
		\iu n\equiv 0 \qquad\text{on the triangle $\overline{OA_nC_n}$.}
	\eeq
\end{claim}
\begin{proof} [Proof of Claim \ref{sub_claim_1}]
	Non-negativity of all iterates is immediate from the scheme in 
	\eq{OAC_soln-0}-\eq{ABC_soln-0}, while \eq{eq-claim1} follows easily 
	by induction on $n$: the base step is provided 
	by the first part of \eq{OAC_soln-0}, and the induction step follows from 
	the second part since $\tau(x,y)\subset \overline{OA_nC_n}$ whenever $(x,y)\in \overline{OA_{n+1}C_{n+1}}$.
\end{proof}
For the base step of the main argument we shall also need the following.
\begin{claim}\label{sub_claim_2}
            For all $n\geq 0$ we have
		\begin{align}\label{eq-claim2}
			\iu {n+1}>\iu n \text{ on the triangle $ABC$.}
 		\end{align}
\end{claim}
\begin{proof} [Proof of Claim \ref{sub_claim_2}]
We argue by induction on $n$. Fix $(x,y)\in ABC$. For $n=0$ we have $\iu 0(x,y)=\theta(y)$, 
which is strictly positive on $T(x,y)$. According to \eq{ABC_soln-0} we thus have
 \begin{align*}
 		\iu 1(x,y)&=\theta(y)+\iint_{T(x,y)}\iu 0(\xi,\eta)\, d\eta d\xi\\
		&=\iu 0(x,y)+ \iint_{T(x,y)}\iu 0(\xi,\eta) d\eta d\xi>\iu 0(x,y).
\end{align*}
Next, assuming that $\iu {n}>\iu {n-1}$ holds on $ABC$ for index $n$, \eq{ABC_soln-0} 
 gives
\begin{align*}
	\iu {n+1}(x,y)&=\theta(y)+\iint_{T(x,y)}\iu n(\xi,\eta)\, d\eta d\xi\\
	&>\theta(y)+ \iint_{T(x,y)}\iu{n-1}(\xi,\eta) d\eta d\xi=\iu n(x,y).
\end{align*}
\end{proof}
\begin{proof} [Proof of Claim \ref{main_claim}]
The proof is by double induction. We set  
\begin{align}
	\tau_0(x,y)&:=\tau(x,y)\cap ABC,\nn\\
	\tau_N(x,y)&:=\tau(x,y)\cap \mathcal T_N
	\qquad\text{for $N\geq 1$.}\label{tau-N}
\end{align} 
{\em Base step $N=1$:} Fix $(x,y)\in \mathcal T_1$ and $n\geq 1$; 
we want to show that  $\iu {n}(x,y)>\iu {n-1}(x,y)$.
This is done by induction on $n$.

{\em Base step $n=1$:} We want show $\iu 1(x,y)>\iu 0(x,y)=0$; we have
\beq\label{eq-iu1}
	\iu1(x,y)=\iint_{\tau(x,y)}\iu 0(\xi, \eta) d\eta d\xi
	=\iint_{\tau_0(x,y)}\theta(\eta) d\eta d\xi>0.
\eeq
{\em Inductive step in $n$:} Assuming that for some $n\geq1$ we have
\beq\label{indh-n1}
	\iu n>\iu{n-1}  \qquad\text{on $\mathcal T_1$,}
\eeq 
we want to show that $\iu {n+1}>\iu{n}$ on $\mathcal T_1$. 
We have
\[\iu{n+1}(x,y)=\iint_{\tau(x,y)}\iu n(\xi, \eta) d\eta d\xi
=\big(\!\!\iint_{\tau_0(x,y)}+\iint_{\tau_1(x,y)}\!\!\big)\iu n(\xi, \eta) d\eta d\xi.\]
On  $\tau_0$, $\iu n(\xi, \eta)> \iu {n-1}(\xi, \eta)$ by Claim \ref{sub_claim_2},
while on $\tau_1$,  $\iu n(\xi, \eta) >\iu {n-1}(\xi, \eta)$, according to the induction 
hypothesis \eqref{indh-n1}. Therefore,
\begin{align*}
	\iu{n+1}(x,y)&>\big(\!\!\iint_{\tau_0(x,y)}+\iint_{\tau_1(x,y)}\!\!\big)\iu {n-1}(\xi, \eta) d\eta d\xi\\
	&=\iint_{\tau(x,y)}\iu {n-1}(\xi, \eta) d\eta d\xi=\iu n(x,y),
\end{align*}
completing the argument for the base step $N=1$.

\noindent
{\em Inductive step in $N$:} We now assume that for some $N\geq 1$ 
and for all  $n\geq N$,
\beq\label{indh-N} 
	\iu {n}>\iu {n-1} \qquad\text{on $\mathcal T_N$},
\eeq
and we want to show that $\iu {n}>\iu {n-1}$ 
on $\mathcal T_{N+1} $ for all $n\geq N+1$. 
This is again done by induction on $n$. 
Fix $(x,y)\in \mathcal T_{N+1} $.

{\em Base step $n=N+1$:}  We want  to show $\iu{N+1}(x,y)>\iu{N}(x,y)$; 
we have
\begin{align*}
	\iu {N+1}(x,y)&=\iint_{\tau(x,y)}\iu {N}(\xi,\eta)\, d\eta d\xi\\
	&= \iint_{\tau_{N+1}(x,y)}\iu{N}(\xi,\eta) d\eta d\xi
	+ \iint_{\tau_N(x,y)}\iu {N}(\xi,\eta) d\eta d\xi,
\end{align*}
where $\tau_N(x,y)$ is 
defined by \eqref{tau-N}. According to the first part of Claim 
\ref{sub_claim_1}, the first term in the sum is non-negative. 
For the second term, we apply the induction hypothesis 
\eqref{indh-N} with $n=N$: 
since $\tau_N(x,y)\subset \mathcal T_N$, we 
have $\iu {N}>\iu {N-1}$ on $\tau_N(x,y)$. Therefore,
\begin{align*}
	\iu {N+1}(x,y)&\geq \iint_{\tau_N(x,y)}\iu {N}(\xi,\eta) d\eta d\xi
	>\iint_{\tau_N(x,y)}\iu {N-1}(\xi,\eta) d\eta d\xi.
\end{align*}
Finally, according to the second part of Claim \ref{sub_claim_1},
$\iu {N-1}$ vanishes on $OA_{N-1}C_{N-1}$, which contains $\tau_{N+1}(x,y)$.
It follows that 
\[\iu {N+1}(x,y)>\iint_{\tau_N (x,y)}\iu {N-1}(\xi,\eta) d\eta d\xi
\equiv \iint_{\tau (x,y)}\iu {N-1}(\xi,\eta) d\eta d\xi=\iu{N}(x,y),\]
which establishes the base step $n=N+1$.

{\em Inductive step in $n$:} Assuming that for some $n\geq  N+1$ we have
\beq\label{indh-n2}
	\iu {n}>\iu{n-1} \qquad \text{on $\mathcal T_{N+1}$,}
\eeq 
we want to show that $\iu {n+1}(x,y)>\iu{n}(x,y)$. We have,
\begin{align*}
	\iu {n+1}(x,y)&=\iint_{\tau(x,y)}\iu {n}(\xi,\eta)\, d\eta d\xi
	= \big(\nquad\iint_{\tau_{N+1}(x,y)}\!\!+ \iint_{\tau_N(x,y)}\!\!\big)
	\iu{n}(\xi,\eta) d\eta d\xi.
\end{align*}
For the first integral we use that 
$\tau_{N+1}(x,y)\subset \mathcal T_{N+1}$
and apply the induction hypothesis \eq{indh-n2} on $n$; for the 
second integral we use that 
$\tau_{N}(x,y)\subset \mathcal T_{N}$, 
and apply the induction hypothesis \eq{indh-N} on $N$.
This yields
\begin{align*}
	\iu {n+1}(x,y)&>\iint_{\tau_{N+1}(x,y)}\iu{n-1}(\xi,\eta) d\eta d\xi
	+ \iint_{\tau_N(x,y)}\iu {n-1}(\xi,\eta) d\eta d\xi\\
	&=\iint_{\tau(x,y)}\iu {n-1}(\xi,\eta)\, d\eta d\xi =\iu {n}(x,y),
\end{align*}
completing the induction on $N$.
\end{proof}
As detailed above, with Claim \ref{main_claim} proved, 
Theorem \ref{main_result} follows.
\end{proof}

\section{Local existence for nonlinear equations}\label{local_existence}
\begin{figure}
	\centering
	\includegraphics[width=11cm,height=7cm]{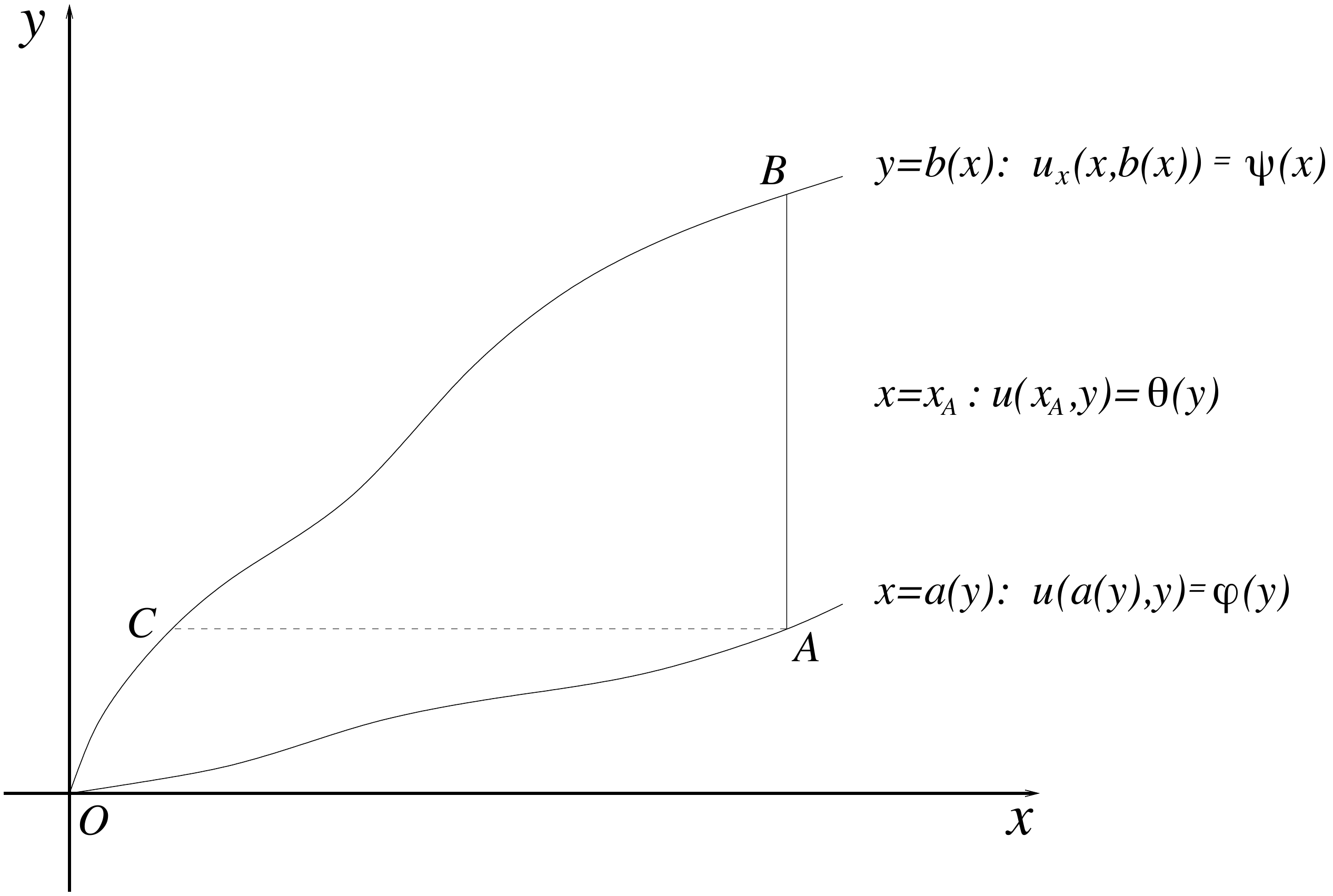}
	\caption{}
	\label{Figure_gen_u_data}
\end{figure}
In this section, we consider nonlinear equations of the form 
\beq\label{eqn1}
	u_{xy}=f(x,y,u,u_x,u_y),
\eeq
with $u$-data prescribed along a curve $M=\{x=a(y)\}$,
and with $u_x$-data prescribed along a curve $N=\{y=b(x)\}$.
We assume that their graphs are located as in Figure~\ref{Figure_gen_u_data}.
More precisely, we consider the following setup:
$A=(x_A,y_A)$ and $B=(x_B,y_B)$ are two points in the first 
quadrant, with $x_A=x_B$ and $y_A<y_B$, and 
\begin{itemize}
	\item[(H1)] $a\in C^1[0,y_A]$, $a$ is strictly increasing, 
	with $a(0)=0$ and $a(y_A)=x_A$;
	\item[(H2)] $b\in C^0[0,x_A]$, $b$ is strictly increasing, 
	with $b(0)=0$ and $b(x_A)=y_B$;
	\item[(H3)] $y<b(a(y))$ for all $y\in(0,y_A]$;
	\item[(H4)] $\vp\in C^1[0,x_A]$ and $\psi \in C^0[0,y_B]$.
\end{itemize}
For simplicity we make the following assumptions on the right-hand member 
$f$ in \eq{eqn1}:
\begin{itemize}
	\item[(H5)]  $f:\overline{OAB}\times \RR^3\to\RR$ is 
	continuous, bounded, and satisfies a uniform Lipschitz condition with
	respect to $v=(u,p,q)$:  {there exists $L>0$, such that}
	\beq\label{lip}
		|f(x,y,v)-f(x,y,\bar v)|\leq L|v-\bar v| 
	\eeq
	for all $(x,y)\in \overline{OAB}$ and all $v,\bar v\in \RR^3$. 
\end{itemize}
Here and below we use the notation
\[|v|=|u|+|p|+|q|\qquad \text{for $v=(u,p,q)\in\RR^3$.}\]
With these assumptions we now pose the following boundary value problem:
\begin{align}
 	u_{xy}(x,y)=f(x,y,u(x,y), u_x(x,y),u_y(x,y)) &
	\qquad\text{for $(x,y)\in\overline{OAB}$} \label{gen_pde}\\
 	u(a(y),y)=\vp(y) &
	\qquad\text{for $y\in [0,y_A]$} \label{gen_udata}\\
	u_x(x,b(x))=\psi(x) &
	\qquad\text{for $x\in [0,x_A]$.}\label{gen_uxdata}
\end{align}
As for the particular case treated in Section \ref{exist_uxy=u}, we shall
obtain a solution of \eq{gen_pde}-\eq{gen_udata}-\eq{gen_uxdata} by 
constructing a solution of the following modified problem:
\begin{align}
 	u_{xy}(x,y)=f(x,y,u(x,y), u_x(x,y),u_y(x,y)) 
	\qquad&\text{for $(x,y)\in\overline{OAB}$} \label{gen_pde'}\\
 	u(a(y),y)=\vp(y) 
	\qquad&\text{for $y\in [0,y_A]$} \label{gen_udata'}\\
	u_x(x,b(x))=\psi(x) 
	\qquad&\text{for $x\in [0,x_A]$}\label{gen_uxdata'}\\
	u(x_A,y)=\theta(y) 
	\qquad&\text{for $y\in [y_A,y_B]$,} \label{gen_udata''}
\end{align}
where $\theta(y)$ is a suitably chosen function defined on $[y_A,y_B]$. 
To obtain a solution of the latter problem, and hence of the original
problem \eq{gen_pde}-\eq{gen_udata}-\eq{gen_uxdata}, we 
will employ Picard iteration. The convergence of the iteration scheme 
will be obtained on a sufficiently small subregion of $\OAB$.

The reformulation of the PDE \eq{gen_pde'}
as an integral equation, and setting up an iteration scheme for the triple
$(u,u_x,u_y)$ is standard; see e.g., Section 21 in \cite{wal}. 
However, there is now an additional ``twist'' to this setup: differently from standard cases
treated in the literature, the choice of $\theta$ is now part of the problem.
Indeed, as we shall see, the presence 
of $u_x$ on the right-hand side of \eq{gen_pde'} forces us to 
consider a scheme which includes iteration of the function $\theta$
as well. 

Before formulating the iteration scheme we analyze the conditions 
that $\theta$ must satisfy in order to yield a classical $C^1$-solution 
to \eq{gen_pde'}-\eq{gen_udata'}-\eq{gen_uxdata'}-\eq{gen_udata''}.
By integrating \eq{gen_pde'}, first with respect to $x$ and then with respect to 
$y$, and making use of the boundary data 
\eq{gen_udata'}-\eq{gen_uxdata'}-\eq{gen_udata''},
we obtain that 
\beq\label{u_OAC}
	u(x,y)=\vp(y)-\int_a^{a(y)}\!\!\!\!\psi(\xi)\, d\xi
	+\int_x^{a(y)}\!\!\!\!\int_y^{b(\xi)}\!\!\!\! f(\xi,\eta,v(\xi,\eta))\, d\eta d\xi
	\quad\text{on $\OAC$,}
\eeq
and
\beq\label{u_ABC}
	u(x,y)=\theta(y)-\int_a^{x_A}\!\!\!\!\psi(\xi)\, d\xi
	+\int_x^{x_A}\!\!\!\!\int_y^{b(\xi)}\!\!\!\! f(\xi,\eta,v(\xi,\eta))\, d\eta d\xi
	\quad\text{on $\ABC$.}
\eeq
As in the proof of Lemma \ref{regularity} we obtain the following: 
in order that the expressions in \eq{u_OAC} and \eq{u_ABC} hold for a 
$C^1(\OAB)$ function $u$, it is necessary that $\theta$ belongs to $C^1[y_A,y_B]$ 
and satisfies
\begin{align}
	\theta(y_A)&=\vp(y_A)\qquad \label{theta_cond1}\\
	\theta'(y_A)&=\vp'(y_A)
	-a'(y_A) \Big[\psi(x_A)-\int_{y_A}^{y_B}\!\!\!\! f(x_A,\eta,v(x_A,\eta))\, d\eta\Big].
	\label{theta_cond2}
\end{align}
We need to make sure that the iteration scheme incorporates 
these conditions. For simplicity we shall use affine functions 
as iterates for $\theta$. We proceed as follows.

\subsection{Iteration scheme} 
\subsubsection{Base step} 
We start by fixing the affine function 
$\ith 0=\vp(y_A)+\sigma_0 (y-y_A)$ characterized by the conditions
\beq\label{cond-th0}
	\ith 0(y_A)=\vp(y_A) \qquad\text{and}\qquad 
	\ith {0}{}^\prime(y_A)=\sigma_0=\vp'(y_A)-{a'(y_A)}   \psi(x_A),
\eeq
and then set
\begin{align}
	\iu 0(x,y)&:=
	\begin{dcases} 
		\vp(y)-\int_x^{a(y)}\!\!  \psi(\xi)d\xi &\qquad \text{for $(x,y)\in \overline{OAC}$}  \\ 
		\ith 0(y)-\int_x^{x_A}\!\! \psi(\xi)d\xi &\qquad \text{for $(x,y)\in \overline{ABC}$,} 
	\end{dcases}\\
	\ip 0 (x,y)&:=\psi(x)\qquad \text{for $(x,y)\in \overline{OAB}$},\\
	\iq 0 (x,y)&:=
	\begin{dcases} 
		\vp'(y)-{a'(y)}   \psi(a(y)) &\qquad \text{for $(x,y)\in \overline{OAC}$}  \\ 
		\ith {0}{}^\prime(y) &\qquad \text{for $(x,y)\in \overline{ABC}$,} 
	\end{dcases}
\end{align}
The conditions in \eqref{cond-th0} ensure that $\iu 0$ and $\iq 0$ are defined consistently 
on $\overline{AC}=\overline{OAC}\cap \overline{ABC}$. It is immediate to verify that the 
following holds:
\begin{align*}
            &\iu 0_x(x,y)=\ip0(x,y)\qquad \text{for $(x,y)\in \overline{OAB}$,} \\
            &\iu 0_y(x,y)=\iq0(x,y)\qquad \text{for $(x,y)\in\overline{OAB}$,}\\
            &\iu 0(a(y),y)=\vp(y)\qquad \text{for $y\in [0,y_A]$,}\\
            &\iu 0(x_A,y)=\ith 0(y)\qquad\text{for $y\in [y_A,y_B]$,}\\ 
            &\ip 0(x, b(x))=\psi(x)\qquad \text{for $x\in [0, x_A]$,}\\
            &\iq 0(x_A, y)=\ith {0}{}^\prime(y)\qquad \text{for $y\in [0, y_A].$}
\end{align*}

\subsubsection{Iteration step}
For $n\geq0$, assume that $\iu n$, $\ip n$, $\iq n$, and $\ith n$ 
are continuous functions on $\overline{OAB}$ that satisfy
\begin{align}
            \label{cond-iunx}&\iu n_x(x,y) =\ip n(x,y)\qquad \text{for $(x,y)\in\overline{OAB}$,}\\
            \label{cond-iuny}&\iu n_y(x,y) =\iq n(x,y)\qquad \text{for $(x,y)\in \overline{OAB}$,}\\
            \label{cond-iun1}&\iu n(a(y),y)=\vp(y) \qquad \text{for $y\in [0,y_A]$,}\\
            \label{cond-iun2}&\iu n(x_A,y)=\ith n(y)\qquad \text{for $y\in [y_A,y_B]$,}\\ 
            \label{cond-ipn} &\ip n(x, b(x))=\psi(x) \qquad \text{for $x\in [0, x_A]$,}\\
            \label{cond-iqn} &\iq n(x_A, y)={\ith n}'(y) \qquad \text{for $y\in [0, y_A]$.}
 \end{align}
We first update the $u$-data along $\AB$ by letting $\ith {n+1}$ be the 
affine function characterized by
\begin{align} 
	\ith {n+1}(y_A)&=\vp(y_A)\label{cond-thn1}\\
	\ith {n+1}{}^\prime(y_A)&=\vp'(y_A)-{a'(y_A)}\Big[\psi(x_A)\nn\\
	&\qquad\qquad-\int_{y_A}^{y_B}\!\! 
	f\big(x_A,\eta, \ith n(\eta), \ip n(x_A,\eta),\ith n {}^\prime(\eta)\big)d\eta\Big].
	\label{cond-thn2}
\end{align}
In accordance with \eq{theta_cond1}-\eq{theta_cond2},
by using this $\ith {n+1}$ in the definitions of $\iu {n+1}$ and $\iq {n+1}$ below,
we guarantee continuity of the next iterate $\iv{n+1}$ across the horizontal 
line $\AC$. As remarked above, we note that the presence of $\ip n(x_A,\eta)$ 
in the integrand on the right-hand side of \eq{cond-thn2} (i.e., the presence of 
$u_x$ in the original PDE \eq{eqn1}), rules out the possibility of 
using a fixed function $\theta$ in all iteration steps.

We then update $\iv n=(\iu n,\ip n,\iq n)$ by setting 
\beq\label{eq-iun1}
	\iu {n+1}(x,y):= \vp(y)-\int_x^{a(y)} \!\!  \psi(\xi)d\xi 
	+\int_x^{a(y)} \!\!\!\! \int_y^{b(\xi)} \!\!  f\big(\xi,\eta,\iv n(\xi,\eta)\big)\, d\eta d\xi
\eeq
for $(x,y)\in \overline{OAC}$,  
\beq\label{eq-iun2} 
	\iu {n+1}(x,y):=   \ith {n+1}(y)-\int_x^{x_A}\!\!  \!\! \psi(\xi)d\xi
	+ \int_x^{x_A}\!\!\!\! \int_y^{b(\xi)}  \!\!\!\! f\big(\xi,\eta,\iv n (\xi,\eta)\big)\, d\eta d\xi
\eeq
for $(x,y)\in \overline{ABC}$;
\beq\label{eq-ipn} 
	\ip{n+1}(x,y):=\psi(x)-\int_y^{b(x)} \!\!\!\!  f\big(x,\eta,\iv n(x,\eta)\big)\, d\eta 
\eeq
for $(x,y)\in \overline{OAB}$;  
\begin{align}\label{eq-iqn1} 
	\iq {n+1}(x,y)&:= \vp'(y)-a'(y)\Big[\psi(a(y)) -\int_y^{b(a(y))} \!\!\!\!  
	f\big(a(y),\eta,\iv n(a(y),\eta)\big)d\eta \Big]\nn \\
	&\quad - \int_x^{a(y)} \!\!\!\!  f\big(\xi, y,\iv n(\xi,y)\big)d\xi
\end{align}
for $(x,y)\in \overline{OAC}$, and 
\beq\label{eq-iqn2}
	\iq {n+1}(x,y):=\ith {n+1}{}^\prime(y)-\int_x^{x_A} \!\!\!\!  f\big(\xi,y,\iv n(\xi,y)\big)\,d\xi
\eeq
for $(x,y)\in \overline{ABC}$.

As noted above, the conditions \eq{cond-thn1} and \eq{cond-thn2} ensure 
continuity of $\iu {n+1}$ and $\iq {n+1}$, respectively, across $\AC$.
It is immediate to verify that, with the definitions above, 
\eqref{cond-iunx}-\eqref{cond-iqn} are satisfied with $n$ replaced by $n+1$.

\subsection{Convergence}
We proceed to establish convergence of the sequence of iterates $\iv n=(\iu n,\ip n,\iq n)$. 
The goal is to show that they form a Cauchy sequence in $C^0(\overline{OAB};\RR^3)$ 
equipped with the norm
\[\|v\|=\|u\|+\|p\|+\|q\|\equiv \sup |u(x,y)|+\sup |p(x,y)|+\sup |q(x,y)|,\]
where the supremums are taken over $(x,y)\in \overline{OAB}$.
This will be established under the condition that the region 
$\OAB$ is sufficiently small. Set
\begin{align*}
	I&:=[y_A,y_B]\\
	l&:=x_A\\
	h&:=y_B\\
	\alpha&:=\text{area}(OAB)\leq l h \\
	\gamma&:=\underset {0\leq y\leq y_A}{\max}|a'(y)|.
\end{align*}
Fix $n\geq 0$. The first step is to estimate the difference between $\ith {n+1}$ and 
$\ith n$ on $[y_A,y_B]$. 
Denote the slope of the affine function $\ith n$ by $\sigma_n$; this  
is given by \eq{cond-th0}${}_2$ and \eq{cond-thn2}. Thus,
\[\ith n(y)=\vp(y_A)+\sigma_n (y-y_A)\qquad n\geq 0.\]
We thus have
\beq\label{ineq-thn}
	|(\ith {n+1}-\ith n)(y)|=|\sigma_{n+1}-\sigma_n|| y-y_A|<h|\sigma_{n+1}-\sigma_n|.
\eeq
For $n\geq 1$, \eqref{cond-thn2} together with the Lipschitz property \eq{lip} give
\begin{align}
	&|\sigma_{n+1}-\sigma_n| 
	=|{a'(y_A)} |\cdot 
	\int_{y_A}^{y_B}\big|f\big(x_A,\eta, \ith n(\eta), \ip n(x_A,\eta),\sigma_n\big)\nn \\
	&\qquad\qquad\qquad\qquad\qquad\qquad\quad
	-f\big(x_A,\eta, \ith {n-1}(\eta), \ip {n-1}(x_A,\eta),\sigma_{n-1}\big)\big|d\eta\nn\\
	&\leq \gamma L\int_{y_A}^{y_B}| (\ith {n}-\ith {n-1})(\eta)|
	+|(\ip {n}-\ip {n-1})(x_A,\eta)|
	+|\sigma_n-\sigma_{n-1}|\, d\eta\nn\\
	&\leq  \gamma Lh 
	\Big[\underset {\eta\in I}{\sup} |(\ith {n}-\ith {n-1})(\eta)|
	+\underset {\eta\in I}{\sup}|(\ip {n}-\ip {n-1})(x_A,\eta)|
	+|\sigma_n-\sigma_{n-1}|\Big]\nn\\
	&\leq  \gamma Lh  \Big[\|\iu {n}-\iu {n-1}\|+\|\ip {n}-\ip {n-1}\|
	+\|\iq {n}-\iq {n-1}\|\Big]\nn\\
	&=\gamma Lh\|\iv {n}-\iv {n-1}\|,\label{ineq-sign}
\end{align}
where for the last inequality we have used \eqref{cond-iun2} and \eqref{cond-iqn}.

For $n\geq 1$ we proceed with similar estimates for $\iu n$, $\ip n$, and $\iq n$. 
From \eqref{eq-iun1}, 
\beq\label{ineq-iun1} 
	|(\iu {n+1}-\iu {n})(x,y)|\leq L\alpha \|\iv {n}-\iv {n-1}\| 
\eeq
for $(x,y)\in \overline{OAC}$, while \eqref{eq-iun2} gives
\begin{align}
	|(\iu {n+1}-\iu {n})(x,y)| &\leq |(\ith {n+1}- \ith {n})(y)| +L \alpha \|\iv {n}-\iv {n-1}\| \nn\\
	&\leq (\gamma L h^2 +L \alpha)  \|\iv {n}-\iv {n-1} \| \label{ineq-iun2} 
\end{align}
for $(x,y)\in \overline{ABC}$,
where in the last inequality we have used \eqref{ineq-thn} and \eqref{ineq-sign}. 
Combining \eqref{ineq-iun1} and \eqref{ineq-iun2}  and using that $\alpha\leq lh$, we get f
\beq\label{ineq-iun} 
	\|\iu {n+1}-\iu {n}\| \leq Lh(\gamma h+l) \|\iv {n}-\iv {n-1} \|.
\eeq
Next, \eqref{eq-ipn} gives
\beq\label{ineq-ipn}
 	|(\ip{n+1}-\ip{n})(x,y)|\leq L h  \|\iv {n}-\iv {n-1} \|
\eeq
for all $(x,y)\in  \overline{OAB}$, \eqref{eq-iqn1} gives
\beq\label{ineq-iqn1} 
	|(\iq {n+1}-\iq {n})(x,y)|\leq Lh\gamma ||\iv {n}-\iv {n-1} || +Ll\|\iv {n}-\iv {n-1}\|
\eeq
for $(x,y)\in\overline{OAC}$, and \eqref{eq-iqn2} gives
\begin{align} 
	|(\iq {n+1}-\iq {n})(x,y)|&\leq | (\sigma_{n+1}- \sigma_n)(y)| +Ll \|\iv {n}-\iv {n-1}\|\nn \\
	 &\leq (\gamma L h +L l)  \|\iv {n}-\iv {n-1} \| \label{ineq-iqn2}
\end{align}
for $(x,y)\in\overline{ABC}$, where in the last inequality we have used \eqref{ineq-sign}.
From \eqref{ineq-iqn1} and \eqref{ineq-iqn2}, we obtain 
\beq\label{ineq-iqn} 
	|(\iq {n+1}-\iq {n})(x,y)|\leq L( \gamma h + l) \|\iv {n}-\iv {n-1}\|. 
\eeq
for all $(x,y)\in \overline{OAB}$.

Finally, combining \eqref{ineq-iun},  \eqref{ineq-ipn}, and  \eqref{ineq-iqn}, we obtain
\[\|\iv {n+1}-\iv {n} ||\leq \mu (L,\gamma,h,l)||\iv {n}-\iv {n-1}\|,\]
where $ \mu (L,\gamma,h,l)=L(2\gamma h+2l+h)$. We then choose $l$, and hence $h$, 
sufficiently small, so that $\mu<1$.  It follows that  the sequence $(\iv n)=(\iu n,\ip n,\iq n)$ 
is Cauchy in $C^0(\overline{OAB};\RR^3)$, and thus converges uniformly to a continuous 
functions $v=(u,p,q):\overline{OAB}\to \RR^3$. As $(\iv n)$ is Cauchy it follows from 
\eq{ineq-thn} that $(\sigma_n)$ is a Cauchy sequence of real numbers. Recalling that the
$\sigma_n$ are the slopes of the affine functions $\ith n:[y_A,y_B]\to \RR$, we obtain that 
$(\ith n)$ is Cauchy in $C^0[y_A,y_B]$; let its limit be $\theta$. As a uniform limit of 
affine functions with value $\vp(y_A)$ at $y=y_A$, $\theta$ is itself affine, 
$\theta(y)=\sigma(y-y_A)+\vp(y_A)$, where 
$\sigma$ is the limit of $(\sigma_n)$.

Thanks to uniform convergences $\iv n\to v$ and $\ith n \to \theta$, sending $n\to\infty$ in 
\eq{cond-thn1}-\eq{eq-iun2}, \eq{eq-ipn}, \eq{eq-iqn1}-\eq{eq-iqn2}
yields the corresponding equations with upper indices $(n)$ and $(n+1)$ removed.
It is an immediate consequence of the resulting identities that $p=u_x$ and 
$q=u_y$ on $\OAB$.
In particular, we obtain that
\beq\label{eq-int1}
	u (x,y)= \vp(y)-\int_x^{a(y)} \!\!\!\! \psi(\xi)d\xi 
	+\int_x^{a(y)} \!\!\!\! \int_y^{b(\xi)} \!\!\!\!  f\left(\xi,\eta,(u,u_x,u_y)(\xi,\eta)\right)d\eta \,d\xi
\eeq
for $(x,y)\in \overline{OAC}$, and
\beq \label{eq-int2}
	u(x,y)=   \theta(y)-\int_x^{x_A}  \!\!\!\! \psi(\xi)d\xi
	+ \int_x^{x_A} \!\!\!\! \int_y^{b(\xi)} \!\!\!\!  f\left(\xi,\eta,(u,u_x,u_y)(\xi,\eta)\right)d\eta \,d\xi
\eeq
for $(x,y)\in \overline{ABC}$. Since $\theta$ satisfies \eq{cond-thn1} and \eq{cond-thn2}  
with upper indices $(n)$ and $(n+1)$ removed, it follows 
from these identities that $u$, $u_x$, $u_y$, and $u_{xy}$
exist and are continuous in $\OAB$, that 
$u_{xy}(x,y)= f\left(x,y,(u,u_x,u_y)(x,y)\right)$ for $(x,y)\in OAB$, $u(y,(a(y))=\vp(y)$
for $y\in[0,y_A]$, and that $u_x(x,b(x))=\psi(x)$ for $x\in[0,x_A]$.
This concludes the proof of local existence for the mixed boundary value problem 
\eq{gen_pde}-\eq{gen_udata}-\eq{gen_uxdata}.

Combining our findings above with those from Section \ref{non_uniqueness} we have 
the following theorem.
\begin{theorem}\label{summing_up}
	With assumptions (H1), (H2), (H3), (H4), and (H5) above, consider the mixed 
	boundary value problem \eq{gen_pde}-\eq{gen_udata}-\eq{gen_uxdata} 
	on the region $\OAB$ as described above; see Figure~\ref{Figure_gen_u_data}.
	Then, for $\OAB$ sufficiently small, this problem has a solution 
	$u$ with the properties that $u$, $u_x$, $u_y$, and $u_{xy}$ are 
	continuous functions on $\overline{OAB}$. The solution is in general 
	not unique. 
\end{theorem}
\vskip5mm
{\bf Acknowledgment:} H.~K.~Jenssen was partially supported by NSF grants
DMS-1311353 and  DMS-1813283. I.~A.~Kogan was partially supported by NSF  DMS-1311743.



\begin{bibdiv}
\begin{biblist}

\bib{azdz}{article}{
   author={Aziz, A. K.},
   author={Diaz, J. B.},
   title={On a mixed boundary-value problem for linear hyperbolic partial
   differential equations in two independent variables},
   journal={Arch. Rational Mech. Anal.},
   volume={10},
   date={1962},
   pages={1--28},
   issn={0003-9527},
   review={\MR{0142909}},
   doi={10.1007/BF00281176},
}
\bib{bjk}{article}{
   author={Benfield, M.},
   author={Jenssen, H.K.},
   author={Kogan, I.A.}
   title={A Generalization of an Integrability Theorem of Darboux},
   journal={J. Geom. Anal.},
   date={2018},
}
\bib{cil}{article}{
   author={Ciliberto, Carlo},
   title={Su alcuni problemi relativi ad una equazione di tipo iperbolico in
   due variabili},
   language={Italian},
   journal={Boll. Un. Mat. Ital. (3)},
   volume={11},
   date={1956},
   pages={383--393},
   review={\MR{0087863}},
}
\bib{ch}{book}{
   author={Courant, R.},
   author={Hilbert, D.},
   title={Methods of mathematical physics. Vol. II},
   series={Wiley Classics Library},
   note={Partial differential equations;
   Reprint of the 1962 original;
   A Wiley-Interscience Publication},
   publisher={John Wiley \& Sons, Inc., New York},
   date={1989},
   pages={xxii+830},
   isbn={0-471-50439-4},
   review={\MR{1013360}},
}
\bib{darb1}{book}{
   author={Darboux, Gaston},
   title={Le\c{c}ons sur la th\'{e}orie g\'{e}n\'{e}rale des surfaces. II},
   language={French},
   series={Les Grands Classiques Gauthier-Villars. [Gauthier-Villars Great
   Classics]},
   note={
   Les congruences et les \'{e}quations lin\'{e}aires aux d\'{e}riv\'{e}es partielles. Les
   lignes trac\'{e}es sur les surfaces. [Congruences and linear partial
   differential equations. Lines traced on surfaces];
   Reprint of the second (1915) edition;
   Cours de G\'{e}om\'{e}trie de la Facult\'{e} des Sciences. [Course on Geometry of the
   Faculty of Science]},
   publisher={\'{E}ditions Jacques Gabay, Sceaux},
   date={1993},
   isbn={2-87647-016-0},
   review={\MR{1324110}},
}
\bib{darb2}{book}{
   author={Darboux, Gaston},
   title={Le\c{c}ons sur la th\'{e}orie g\'{e}n\'{e}rale des surfaces. IV},
   language={French},
   series={Les Grands Classiques Gauthier-Villars. [Gauthier-Villars Great
   Classics]},
   note={
   D\'{e}formation infiniment petite et repr\'{e}sentation sph\'{e}rique. [Infinitely
   small deformation and spherical representation];
   Reprint of the 1896 original;
   Cours de G\'{e}om\'{e}trie de la Facult\'{e} des Sciences. [Course on Geometry of the
   Faculty of Science]},
   publisher={\'{E}ditions Jacques Gabay, Sceaux},
   date={1993},
   isbn={2-87647-016-0},
   review={\MR{1365962}},
}
\bib{dib}{book}{
   author={DiBenedetto, Emmanuele},
   title={Partial differential equations},
   series={Cornerstones},
   edition={2},
   publisher={Birkh\"{a}user Boston, Inc., Boston, MA},
   date={2010},
   pages={xx+389},
   isbn={978-0-8176-4551-9},
   review={\MR{2566733}},
   doi={10.1007/978-0-8176-4552-6},
}
\bib{gar}{book}{
   author={Garabedian, P. R.},
   title={Partial differential equations},
   note={Reprint of the 1964 original},
   publisher={AMS Chelsea Publishing, Providence, RI},
   date={1998},
   pages={xii+672},
   isbn={0-8218-1377-3},
   review={\MR{1657375}},
}
\bib{gour}{book}{
   author={Goursat, \'{E}douard},
   title={Cours d'analyse math\'{e}matique. Tome III},
   language={French},
   series={Les Grands Classiques Gauthier-Villars. [Gauthier-Villars Great
   Classics]},
   edition={3},
   note={Int\'{e}grales infiniment voisines. \'{E}quations aux d\'{e}riv\'{e}es du second
   ordre. \'{E}quations int\'{e}grales. Calcul des variations. [Infinitely near
   integrals. Second-order partial differential equations. Integral
   equations. Calculus of variations]},
   publisher={\'{E}ditions Jacques Gabay, Sceaux},
   date={1992},
   pages={v+704},
   isbn={2-87647-032-2},
   review={\MR{1296666}},
}
\bib{khar}{article}{
   author={Kharibegashvili, S.},
   title={Goursat and Darboux type problems for linear hyperbolic partial
   differential equations and systems},
   language={English, with English and Georgian summaries},
   journal={Mem. Differential Equations Math. Phys.},
   volume={4},
   date={1995},
   pages={127},
   issn={1512-0015},
   review={\MR{1415805}},
}
\bib{kj}{article}{
   author={Kharibegashvili, S. S.},
   author={Jokhadze, O. M.},
   title={On the solvability of a boundary value problem for nonlinear wave
   equations in angular domains},
   note={Translation of Differ. Uravn. {\bf 52} (2016), no. 5, 665--686},
   journal={Differ. Equ.},
   volume={52},
   date={2016},
   number={5},
   pages={644--666},
   issn={0012-2661},
   review={\MR{3541458}},
   doi={10.1134/S0012266116050104},
}
\bib{lieb}{book}{
   author={Lieberstein, H. Melvin},
   title={Theory of partial differential equations},
   note={Mathematics in Science and Engineering, Vol. 93},
   publisher={Academic Press, New York-London},
   date={1972},
   pages={xiv+283},
   review={\MR{0355280}},
}
\bib{nak}{article}{
   author={Nakhushev, A. M.},
   title={Goursat problem},
   journal={Encyclopedia of Mathematics, 
   https://www.encyclopediaofmath.org/index.php/Goursat${}_-{}$problem, accessed May 2019},
}
\bib{ri_coll}{collection}{
   author={Riemann, Bernhard},
   title={Collected papers},
   note={Translated from the 1892 German edition by Roger Baker, Charles
   Christenson and Henry Orde},
   publisher={Kendrick Press, Heber City, UT},
   date={2004},
   pages={x+555},
   isbn={0-9740427-2-2},
   isbn={0-9740427-3-0},
   review={\MR{2121437}},
}
\bib{pic}{book}{
   author={Picard, \'{E}mile},
   title={Le\c{c}ons sur quelques \'{e}quations fonctionnelles avec des applications
   \`a divers probl\`emes d'analyse et de physique math\'{e}matique. R\'{e}dig\'{e}es par
   Eug\`ene Blanc},
   language={French},
   publisher={Gauthier-Villars, Paris},
   date={1950},
   pages={iii+187},
   review={\MR{0033960}},
}
\bib{pog}{book}{
   author={Pogorzelski, W.},
   title={Integral equations and their applications. Vol. I},
   series={Translated from the Polish by Jacques J. Schorr-Con, A. Kacner
   and Z. Olesiak. International Series of Monographs in Pure and Applied
   Mathematics, Vol. 88},
   publisher={Pergamon Press, Oxford-New York-Frankfurt; PWN-Polish
   Scientific Publishers, Warsaw},
   date={1966},
   pages={xv+714pp},
   review={\MR{0201934}},
}
\bib{szm1}{article}{
   author={Szmydt, Z.},
   title={Sur l'existence de solutions de certains nouveaux probl\`emes pour
   un syst\`eme d'\'{e}quations diff\'{e}rentielles hyperboliques du second ordre \`a
   deux variables ind\'{e}pendantes},
   language={French},
   journal={Ann. Polon. Math.},
   volume={4},
   date={1957},
   pages={40--60},
   issn={0066-2216},
   review={\MR{0094568}},
   doi={10.4064/ap-4-1-40-60},
}
\bib{szm2}{article}{
   author={Szmydt, Z.},
   title={Sur l'existence d'une solution unique de certains probl\`emes pour
   un syst\`eme d'\'{e}quations diff\'{e}rentielles hyperboliques du second ordre \`a
   deux variables ind\'{e}pendantes},
   language={French},
   journal={Ann. Polon. Math.},
   volume={4},
   date={1958},
   pages={165--182},
   issn={0066-2216},
   review={\MR{0094569}},
   doi={10.4064/ap-4-2-165-182},
}
\bib{wal}{book}{
   author={Walter, Wolfgang},
   title={Differential and integral inequalities},
   series={Translated from the German by Lisa Rosenblatt and Lawrence
   Shampine. Ergebnisse der Mathematik und ihrer Grenzgebiete, Band 55},
   publisher={Springer-Verlag, New York-Berlin},
   date={1970},
   pages={x+352},
   review={\MR{0271508}},
}

\end{biblist}
\end{bibdiv}
\end{document}